\documentclass[leqno,12pt]{article}
\usepackage{amsmath,amsthm,amssymb}
\usepackage[T1]{fontenc}

\newtheorem{Th}{Theorem}[section]
\newtheorem{Cor}[Th]{Corollary}
\newtheorem{Lemma}[Th]{Lemma}
\newtheorem{Prop}[Th]{Proposition}

\theoremstyle{definition}

\theoremstyle{remark}
\newtheorem{Remark}{Remark}

\newtheorem{Def}{Definition}

\newcommand{\al}{\alpha}

\newcommand{\tor}{{\Bbb{T}}}

\newcommand{\ce}{\E}

\newcommand{\cb}{{\cal B}}

\newcommand{\bs}{\mathbb{S}}

\newcommand{\Q}{\mathbb{Q}}
\newcommand{\R}{{\mathbb{R}}}
\newcommand{\T}{{\mathbb{T}}}
\newcommand{\C}{{\mathbb{C}}}
\newcommand{\Z}{{\mathbb{Z}}}

\newcommand{\xbm}{(X,{\cal B},\mu)}

\newcommand{\beq}{\begin{equation}}
\newcommand{\eeq}{\end{equation}}

\newcommand{\qnk}{q_{n_{k}}}
\newcommand{\qn}{q_{n}}

\newcommand{\E}{{\cal E}}

\begin{document}

\title{On multiple ergodicity of affine
cocycles \\ over irrational rotations}

\author{Jean-Pierre Conze and Agata Pi\c{e}kniewska}

\maketitle

\begin{abstract}
Let $T_\alpha$ denote the rotation $T_{\alpha}x=x+\alpha$ (mod~$1$)
by an irrational number $\alpha$ on the additive circle $\T=[0,1)$.
Let $\beta_1,..., \beta_{d}$ be $d\geq 1$ parameters in $[0, 1)$.
One of the goals of this paper is to describe the ergodic properties
of the cocycle (taking values in $\R^{d+1}$) generated over
$T_\alpha$ by the vectorial function $\Psi_{d+1}(x):=(\varphi(x),
\varphi(x+\beta_1),..., \varphi(x+\beta_{d})), {\rm \ with \ }
\varphi(x)=\{x\}-\frac12.$

It was already proved in \cite{LeMeNa03} that $\Psi_{2}$ is regular
for $\alpha$ with bounded partial quotients. In the present paper we
show that $\Psi_{2}$ is regular for any irrational $\alpha$. For
higher dimensions, we give sufficient conditions for regularity.
While the case $d=2$ remains unsolved, for $d=3$ we provide examples
of non-regular cocycles $\Psi_{4}$ for certain values of the
parameters $\beta_1,\beta_2,\beta_3$.

We also show that the problem of regularity for the cocycle $\Psi_{d+1}$ reduces to
the regularity of the cocycles of the form $\Phi_{d} =(1_{[0,
\beta_j]} - \beta_j)_{j= 1, ..., d}$ (taking values in $\R^d$). Therefore, a large
part of the paper is devoted to the classification problems of step
functions with values in $\R^{d}$.
\end{abstract}

\tableofcontents \thispagestyle{empty}

\section{Introduction}

Denote by $\T=[0,1)$ the additive circle. Let $\alpha \in (0,1)$ be
an irrational number and $T_\alpha$ stand for the corresponding
circle rotation: $T_\alpha x=x+\alpha$. The meaning of $\alpha$
being fixed, throughout the paper, we will write $T$ instead of
$T_\alpha$ (except for Sections 2.2 and 2.3).

Let $\beta_1,..., \beta_{d}$ be $d\geq 1$ parameters in $[0, 1)$. We
consider the cocycle generated over $T$ by the vectorial function
\begin{eqnarray}\label{pierwsza}
\Psi_{d+1}(x):=(\psi(x), \psi(x+\beta_1),..., \psi(x+\beta_{d})),
{\rm \ with \ } \psi(x)=\{x\}-\frac12. \label{phid}
\end{eqnarray}

This cocycle takes values in $\R^{d+1}$ and one of the goals of this
paper is to describe its ergodic properties. Namely, we are mostly
interested whether or not $\Psi_{d+1}$ is regular (that is,
cohomologous to a ``smaller'' cocycle which is ergodic, see
Section~\ref{prelimin} for the precise meaning of regularity). It is
well known that $\Psi_1:\T\to\R$ is ergodic for each irrational
$\alpha$, but for $d\geq1$ the problem of regularity is unsolved. As
for applications in ergodic theory, or more precisely in the theory
of joinings, an importance of regularity of cocycles of the
form~(\ref{pierwsza}) has been shown in~\cite{LeMeNa03}. Indeed,
Theorem~3 therein gives the full description of all ergodic
self-joinings for so called Rokhlin extensions given by regular
cocycles. In particular, it is shown in \cite{LeMeNa03} that
$\Psi_{2}$ is regular whenever $\alpha$ has bounded partial
quotients. In the present paper we show that $\Psi_{2}$ is regular
without the assumption of boundedness on the partial quotients of
$\alpha$. For higher dimensions ($d\geq2$), we give sufficient
conditions for regularity. While the case $d=2$ we leave unsolved,
for $d=3$ we give examples of non-regular cocycles $\Psi_{4}$ for
certain values of the parameters $\beta_j$, $j=1,2,3$. In
Section~\ref{secgeneric}, we show that the cocycle~(\ref{pierwsza})
is ergodic for a generic choice (in the measure-theoretic and the
topological sense) of parameters $\beta_1,\ldots,\beta_d$, $d\geq2$.

One of our basic tools is Theorem~\ref{reduc1} below. It states that
the group of essential values of $\Psi_{d+1}$ contains the diagonal
subgroup $\Delta_{d+1}=\{(t,...,t):\: t \in \R\}\subset\R^{d+1}$. It
follows that the problem of regularity for the cocycle $\Psi_{d+1}$
is reduced to the regularity of the cocycles of the form $\Phi_{d}
=(1_{[0, \beta_j)} - \beta_j, j= 1,...,d)$. Note that by taking
linear combinations of cocycles of the form $\Phi_d$ we can get
every step cocycle. Therefore, we devote a large part of the paper
to the problem of classification of steps functions with values in
$\R^{d}$. The problem of ergodicity or regularity of step functions,
mainly in the one dimensional case, has been broadly studied in the
literature, for instance see: \cite{Or83}, \cite{Pa90}, \cite
{Fr00}, \cite{LePaVo96}, \cite{Co09}, \cite{Zh10}; note that in
Corollary~\ref{betawsd} we generalize the main result of
\cite{Zh10}. We would like to emphasize that the methods presented
in the paper, in large part (see Section~\ref{nowa}), seem to be new
and they contribute to a better understanding of the problem of
regularity of general vectorial cocycles $\Phi:\T\to\R^d$ over
irrational rotations.
\section{Preliminaries}

\subsection{Irrational rotations} \label{subsectIrr}
Let us recall some basic facts about continued fractions (e.g.\
\cite{Kh}). Let $[0;a_1,..., a_n,...]$ be the continued fraction
representation of $\alpha$, and let $(p_n/q_n)_{n \ge -1}$ be the
sequence of its convergents. The integers $p_n$ (resp.\ $q_n$) are
the {\em numerators} (resp.\ {\em denominators}) of $\alpha$. We
have $p_{-1}=1$, $p_0=0$, $q_{-1}=0$, $q_0=1$, and for $n \ge 1$:
\begin{equation} \label{converg_eq}
p_n = a_n p_{n-1}+p_{n-2}, q_n = a_n q_{n-1}+q_{n-2}, \ (-1)^n =
p_{n-1} q_n - p_n q_{n-1}.
\end{equation}

As usual the fractional part of $u \in \R$ is $\{u \}= u - [u]$,
where $[u]$ is the integral part of $u$. For $u\in\R$, set $\|u\|=
\inf_{n \in \Z} |u - n| = \min (\{u \}, 1 - \{u\})$. Then for any
integer $M$ we have $\|Mu\|\leqslant |M| \|u\|$. Note that $\|
\cdot\|$ introduces a translation invariant distance on $\tor$.

We have for $n \ge 0$, $\|q_n \alpha\| = (-1)^n \theta_n$ with
$\theta_n=q_n \alpha - p_n$, and moreover
\begin{eqnarray}
1 &=& q_n\|q_{n+1} \alpha \| + q_{n+1} \|q_n \alpha\|, \label{f_1} \\
{1\over q_{n+1}+q_n} &\le& \|q_n \alpha\| \le {1\over q_{n+1}}
= {1 \over a_{n+1} q_n+q_{n-1}}, \label{f_3} \\
\|q_n \alpha \| &\leq& \|k \alpha \|, \ \mbox{for}\ 1 \le |k| <
q_{n+1} \label{f_4}.
\end{eqnarray}

An irrational $\alpha$ is said to be of {\em bounded type} if the
sequence $(a_n)$ is bounded.

We need some preliminary lemmas on the diophantine properties of
$\alpha$.

\begin{Lemma}\label{cont2qn}
1) Let $p,q$ be two coprime positive integers and $\theta =q(\alpha
-{p \over q})$ with $|\theta| < {1 \over q}$. When $\theta > 0$,
each interval $[{j \over q}, {j+1 \over q})$, $0 \leq j \leq q-1$
contains one (and only one) number of the form $\{k \alpha\}$, with
$0 \leq k \leq q-1$. When $\theta < 0$ the same is true for $j=1,
\ldots, q-2$; there are two points $k\alpha$ (one for $k=0$) in $[0,
{1\over q})$ and no such a point in $[{q-1\over q},1)$.

2) For each $x\in\T$ the distance between two consecutive elements
of the set $\big\{\{x + k\alpha\}:\: k = 0,\dots, q-1\big\}$ is $<
{2 \over q}$.

3) There are at most two elements of the set $\big\{\{x + k
\alpha\}:\: k = 0, \dots, q-1\big\}$ in any interval on the circle
of length ${1 \over q}$ (hence at most four such elements in any
interval of length ${2 \over q}$).

4) If additionally $q=q_n$, the distance between two consecutive
elements of the set $\big\{\{x - k \alpha\}:\: k = 0,\dots, q\big\}$
is $> {1 \over 2q_n}$.

\end{Lemma}
\begin{proof} The map $k \to k p {\rm \ mod \ } q = j(k) $ is a bijection of
$\{0, 1, ..., q-1\}$ onto itself. If $\theta > 0$, then $\{k
\alpha\} = \{k ({p \over q} + {\theta \over q}) \} = {j(k) \over q}
+ {k \theta \over q}$ is at distance ${k \theta \over q} < {1 \over
q}$ from ${j(k) \over q}$, hence it is in the interval $[{j(k) \over
q}, {j(k) + 1 \over q})$. The proof is similar if $\theta < 0$.
Hence the first assertion follows.

Assertion 2) is true for $x=0$ by 1); hence, because the distance is
invariant by translations, it is true for any $x\in\T$.

For 3), suppose that there are $\{x + k_1 \alpha\} < \{x + k_2
\alpha\} < \{x + k_3 \alpha\}$ distinct in an interval of length $<
1/q$. We have ${\ell\over q} \leq \{k_1 \alpha\} < \{ k_2 \alpha\} <
\{k_3 \alpha\} < {\ell + 2 \over q}$, for some $\ell$. Either
$[{\ell \over q}, {\ell+1 \over q})$ or $[{\ell +1 \over q}, {\ell
+2 \over q})$ contains two points of the set $\big\{\{k \alpha\}:\:
0 \leq k < q-1\big\}$, which clearly contradicts 1).

4) We have the following $${1 \over 2q_n} \leq {1 \over q_n +
q_{n-1}} \leq \|q_{n-1} \alpha\| \leq \|j \alpha\|, \ \forall j, |j|
< q_n$$ and the assertion follows.
\end{proof}

The first assertion of Lemma~\ref{cont2qn} implies easily the
well-known {\em Denjoy-Koksma inequality}: let $\varphi$ be a
centered function of bounded variation $V(\varphi)$ and $p/q$ a
rational number (in lowest terms) such that $|\alpha - p/q| < {1 /
q^2}$, then
\begin{equation} \label{f_8}
\left|\sum_{\ell = 0}^{q-1}\varphi(x+\ell \alpha)\right| \le
V(\varphi).
\end{equation}

Indeed, let us consider the case $\theta > 0$ (the proof is
analogous when $\theta < 0$). We can assume $x =0$. For $j= 0, ...,
q-1$, there is one and only one point $\{k_j \alpha\}$ of the set
$(\{k \alpha\}, k= 0, ..., q-1)$ in $I_j:=[{j \over q}, {j+1 \over
q})$. Therefore, since $\int \varphi \, dt = 0$, we have:
\begin{eqnarray*}
|\sum_{j=0}^{q-1} \varphi(j \alpha)| &&= |\sum_{j=0}^{q-1}
\varphi(\{k_j \alpha\}) - q\int \varphi(t) \, dt| =
|\sum_{j=0}^{q-1}  q \int_{j/q}^{(j+1) /q}
[\varphi(\{k_j \alpha\}) - \varphi(t)] dt| \\
&&\leq  q \sum_{j=0}^{q-1} \int_{j/q}^{(j+1)/q} |\varphi(\{k_j
\alpha\}) - \varphi(t)| \, dt \leq \sum_{j=0}^{q-1} Var(\varphi,
I_j) = Var(\varphi, [0, 1)).
\end{eqnarray*}

\begin{Lemma} \label{orbi1} 1) If there exists $n_0$ such that
$\inf_{0 \leq |j| < q_n} \|\beta - j \alpha\| < \frac12
\|q_{n+1}\alpha\|, \ \forall n \ge n_0$, then $\beta \in \Z \alpha +
\Z$.

2) Suppose $\alpha$ of bounded type.

a) If $\beta \not \in \Z \alpha + \Z$, then there exist $c >0$ and
an increasing sequence $(n_k)$ such that, for every $k \geq 1$,
$\|\beta - j \alpha\| \geq c/q_{n_k}$, for $0 \leq |j| \leq
q_{n_k}$.

b) If $\beta = {t\over r} \alpha + {u\over s} \in
(\Q\alpha+\Q)\setminus (\Z\alpha+\Z)$, then there exists $c
> 0$ such that $\|\beta - j \alpha\| \geq c/q_{n}$, for $0 \leq |j|
\leq q_{n}$ $(n\geq1)$.
\end{Lemma}
\begin{proof} 1) For each $n \geq 1$, consider the family of intervals
$I_n^j=[\{j\alpha\} -\frac12 \|q_{n+1} \alpha\|, \{j\alpha\}+\frac12
\|q_{n+1} \alpha\|]$, $j= -q_n+1, ..., q_n-1$.

Let $n \geq n_0$. If $j \in \{-q_n+1,...,q_n -1\}$ and $j' \in
\{-q_{n+1} + 1,...,q_{n+1} - 1\}$ are distinct, then the intervals
$I_n^j$ and $I_{n+1}^{j'}$ are disjoint, since otherwise by
 $\|(j'-j)\alpha\| \leq \frac12 \|q_{n+1} \alpha\| +
\frac12 \|q_{n+2} \alpha\| < \|q_{n+1} \alpha\|$, with $0 < |j'-j| <
q_{n} + q_{n+1} \leq q_{n+2}$ which contradicts (\ref{f_4}).

If $\inf_{0 \leq |j| < q_n} \|\beta - j \alpha\| < \frac12
\|q_{n+1}\alpha\|$ for $n \ge n_0$, then there is a sequence
$(j_n)_{n \geq n_0}$ such that $0 \leq |j_n| < q_n$ and $\beta \in
I_n^{j_n}$ for $n \geq n_0$.

Since $\beta\in I^{j_n}_n\cap I^{j_{n+1}}_{n+1}$, we have
$j_0:=j_{n_0}=j_{n_1}=...$. This implies $\beta=\{j_0\alpha\}$ which
completes the proof of 1).

\vskip 3mm 2a) By part 1) if $\beta \not \in \Z \alpha + \Z$, it
follows that there exists a sequence $(n_k)$ such that $\|\beta - j
\alpha\| \geq \frac12 \|q_{n_k+1}\alpha\|$, for $0 \leq |j| \leq
q_{n_k}$ and $k \geq 1$. Suppose additionally that $\alpha$ is of
bounded type. Since $\|q_{n_k+1}\alpha\|$ and ${1\over q_{n_k}}$ are
comparable, there is $c >0$ such that $\|\beta - j \alpha\| \geq
c/q_{n_k}$ for $0 \leq |j| \leq q_{n_k}$.

\vskip 3mm 2b) Now let $\beta = {t \over r} \alpha +{u \over s} \not
\in \Z \alpha + \Z$ with $t, r, u, s$ integers and $r, s \geq 1$.
Let $j_n$ be such that $\varepsilon_n := \min_{j: 0 \leq |j| \leq
q_{n}} \|{t \over r} \alpha +{u \over s} - j \alpha\| = \|{t \over
r} \alpha +{u \over s} - j_n \alpha\|>0$.

We have ${t \over r} \alpha +{u \over s}= j_n \alpha + \ell_n \pm \
\varepsilon_n$, for an integer $\ell_n$; hence: $(r s j_n - ts)
\alpha = r u - r s \ell_n \pm \ r s\varepsilon_n$. It follows
\begin{eqnarray}
\|(r s j_n - ts) \alpha\| \leq r s \, |\varepsilon_n|.
\label{varepn}
\end{eqnarray}

Suppose that $r s j_n - ts = 0$ for infinitely many $n$. Then ${t
\over r} = j_n$ and $|u-s \ell_n| = s |\varepsilon_n|$. Since
$|\varepsilon_n|$ is arbitrarily small for $n$ large enough and $u,
s, \ell_n$ are integers, it follows $u= s \ell_n$. Then, we find
$\beta = j_n \alpha + \ell_n$ contrary to the assumption that
$\beta$ is not in $\Z \alpha + \Z$. It follows that the integers $r
s j_n - ts$ are different from zero for all $n\geq n_1$.

Now, $\alpha$ is of bounded type, so there is $K>0$ such that
$q_{n+rs+1} \leq K \, q_{n}$, for every $n \geq 1$. Using
additionally (\ref{f_3}) and (\ref{f_4}), we obtain
\begin{eqnarray}
{1\over 2 K \, q_{n}} \leq {1\over 2 q_{n+rs+1}} \leq \|q_{n+rs \,}
\alpha \| \leq \|k \alpha \|, \ \mbox{for}\ 1 \le |k| < q_{n+rs+1}.
\label{over2K}
\end{eqnarray}

On the other hand, in view of~(\ref{converg_eq}), given any constant
$C>0$ we have \beq\label{dodaugust}q_{n+m}\geq mq_n+C\eeq for all
$m\geq1$ and $n$ large enough (indeed, it suffices to consider $n$
so that $q_{n-1}\geq C$). Hence, for the integer $|r s j_n - ts|$ we
have
$$0<|r s j_n - ts| \leq rsq_{n} +|t|s \leq q_{n+rs+1}$$ whenever
$n$ is large enough. Therefore, for $n$ large enough, by
(\ref{varepn}) and (\ref{over2K}), we obtain
$$|\varepsilon_n| \geq c / q_n, \text { with } c = {1 \over 2 K}.$$
By taking $c>0$ smaller if necessary, the conclusion holds for all $n\geq1$.
\end{proof}

\vskip 3mm
\begin{Lemma} \label{orbi2} Suppose $\alpha$ of bounded type. Let $B$
be a non-empty finite subset of $(\Q \beta + \Q \alpha + \Q) -
(\Z\alpha + \Z)$, where $\beta$ is a real number. Then there exist
$c > 0$ and a strictly increasing sequence $(n_k)$ such that
$$\forall \beta_i \in B, \ \forall k \geq 1, \ \|\beta_i - j \alpha\|
\geq c/q_{n_k}, \text{ for } 0 \leq |j| \leq q_{n_k}.$$
\end{Lemma}
\begin{proof} We have $B=B_0 \cup B_1$, where the elements $\beta_i$
in $B_0$ are of the form $\beta_i= {u_i \over s_i} \alpha + {w_i
\over s_i}$, with $u_i, w_i, s_i$ integers, $s_i \not= 0$, $\beta_i
\not \in \Z \alpha + \Z$, and the elements in $B_1$ of the form
$\beta_i= {v_i \over s_i} \beta + {u_i \over s_i} \alpha + {w_i
\over s_i}$, with $v_i, u_i, w_i, s_i$ integers and $v_i, s_i \not=
0$. Remark that $B_0$ or $B_1$ can be empty and that $B= B_0$ if
$\beta \in \Q \alpha + \Q$.

If $\beta \not \in \Q\alpha + \Q$ and $B_1$ is not empty, we apply
Lemma~\ref{orbi1} to $\beta':= (\prod v_\ell) \beta$. Let $M= (\max
s_\ell) (\prod v_\ell)$, $M_i= s_i\prod_{\ell \not = i} v_{\ell}$.
We have ${v_i \over s_i}\beta = {\beta' \over M_i}$. There are a
positive constant $c$ and a sequence $(n_k)$ such that
$$\|\beta' - j \alpha \| \geq {c \over q_{n_k}}, \ 0 \leq |j| <
q_{n_k}.$$

Since $L_i := M_i {u_i \over s_i}$ and $M_i {w_i \over s_i}$ are
integers, we have for $j$ such that $0 \leq |M_i j - L_i| <
q_{n_k}$:
\begin{eqnarray*}
M_i \left\|{v_i \over s_i}\beta + {u_i \over s_i} \alpha + {w_i
\over s_i} - j \alpha \right\| &\geq& \left\|M_i{v_i \over s_i}\beta
- M_i\left(j\alpha - {u_i
\over s_i} \alpha - {w_i \over s_i}\right) \right\| \\
&=& \left\|\beta' - (M_i j - L_i) \alpha\right\| \geq {c \over
q_{n_k}}.
\end{eqnarray*}
We have $M_i |j|+ |L_i| \leq M |j| + L$, with $L:= \max |L_i|$. As
$\alpha$ is of bounded type, there are $r$ and $K$ such that $M
q_{n-r} +L \leq q_n \leq K q_{n-r}$, for all $n \geq 1$. This
implies, simultaneously for every $i$:
$$\left\|{v_i \over s_i}\beta + {u_i \over s_i} \alpha + {w_i \over
s_i} - j \alpha \right\| \geq {1 \over M_i} \left\|\beta' - (M_i j -
L_i) \alpha\right\| \geq {c \over MK} {1 \over q_{n_k-r}}, \text{ if
} |j| < q_{n_k -r}.$$

For $\beta_i$ in $B_0$, if this subset is non empty, by the part 2b)
of the previous lemma any subsequence of $(q_n)$ is ``good''.

We conclude that the subsequence $(q_{{n_k}-r})_{r \geq 1}$ fulfills
the assertion of the lemma.
\end{proof}

\vskip 3mm
\begin{Remark} As the proof of Lemma~\ref{orbi2} shows, the
result is true for any change of the part belonging to $\Q\alpha+\Q$
for the elements of $B_1$ (that is, we may replace $ {u_i \over s_i}
\alpha + {w_i \over s_i}$, for $i=1,...,t$, by a different element
of $\Q\alpha+\Q$). However, each time we change this part, we also
change the resulting subsequence $(q_{n_k})$.\end{Remark}

\begin{Remark} When $\alpha$ is not of bounded type, the set
$K(\alpha) = \{\beta\in\R: \lim_n \|q_n \beta\| = 0 \}$ is an
uncountable additive subgroup of $\R$.

Nevertheless, if $\lim_n \|q_n \beta\| = 0$ and $\beta \not \in \Z
\alpha + \Z$, the rate of convergence toward 0 is moderate, as shown
by the following lemma (see \cite{Co80}, \cite{La88}, \cite{KrLi91},
\cite{Co09}).
\end{Remark}

\begin{Lemma} \label{qnbeta0} If there exists $n_0$ such that
$\|q_n \beta\|\le {1\over 4} q_n \|q_n \alpha\|$ for $n \ge n_0$,
then $\beta \in \Z \alpha + \Z$. In particular, if $\alpha$ is of
bounded type and $\beta$ satisfies $\lim_n \|q_n \beta \| = 0$, then
$\beta \in \Z \alpha + \Z$.
\end{Lemma}

\subsection{Essential values of cocycles taking values in Abelian groups}
\label{prelimin}

In this subsection we recall the definition and general results
about essential values of a cocycle (see \cite{Sc77}, \cite{Aa97}).

Let $\xbm$ be a (non-atomic) standard Borel probability space and
$T:\xbm\rightarrow\xbm$ an ergodic automorphism. Such an
automorphism is then automatically aperiodic (that is, for each
$n\geq1$, $\{x\in X:\:T^nx=x\}$ has measure zero).

Assume that $G$ is an Abelian locally compact second countable
(l.c.s.c.) group with the $\sigma$-algebra of its Borel sets
$\cb(G)$ and a fixed Haar measure $m_G$ (we will also write $dg$
instead of $m_G$). Denote by $\overline{G}=G\cup\{\infty\}$ the
one-point compactification of $G$ (when $G$ is non compact).

 For a measurable function $\varphi: X \rightarrow G$, we denote by
$(\varphi_n)$ the cocycle generated by $\varphi$: $$\varphi_n(x) =
\sum_{k=0}^{n-1} \varphi(T^k x), \, n \geq 1$$ and we extend the
formula to all $n\in\Z$ so that, for $n , k \in \Z$,
$\varphi_{n+k}(x)=\varphi_n(x)+\varphi_k(T^nx)$.
 For simplicity, the function $\varphi$ itself will be
called a {\it cocycle}. We say that a cocycle $\varphi : X
\rightarrow G$ is ergodic if the transformation $T_\varphi: (x, g)
\to (Tx, g + \varphi(x))$ is ergodic on $X \times G$ for the measure
$\mu \times dg$.

\vskip 3mm {\it Recurrence of a cocycle}

Let $\| \ \|$ be a norm on $\R^d$. The inequality
$|\|\varphi_{n+1}(x)\| -  \|\varphi_n(T x)\|| \leq \|\varphi(x)\|$
implies the $T$-invariance of the set $\{x: \lim_n \|\varphi_n(x)\|
= +\infty\}$. Therefore by ergodicity this set has measure 0 or 1,
and  we have the following alternative: either for $\mu$-a.e. every
$x$, $\lim_n \|\varphi_n(x)\| = +\infty$ or for $\mu$-a.e. $x$
$\liminf_n \|\varphi_n(x)\| < +\infty$.

\begin{Def} A cocycle $(\varphi_n)$ over $(X, \mu, T)$ with values in $G = \R^d$
is {\it recurrent} if $\liminf_n \|\varphi_n(x)\| < \infty$, for
a.e. $x \in X$. It is {\it transient} if $\lim_n \|\varphi_n(x)\| =
+\infty$, for a.e. $x \in X$.
\end{Def}
It can be shown that recurrence for $(\varphi_n)_{n \ge 0}$ is
equivalent to conservativity of $T_\varphi$ with respect to $\mu
\times dg$ and that it implies $\liminf_n \|\varphi_n(x)\| = 0$ for
a.e. $x$.

It is also equivalent to  the following property: for each
neighborhood $U\ni0$ and $A\subset X$ of positive measure there
exists $N\in\Z\setminus\{0\}$ such that \beq\label{recu}\mu(A\cap
T^{-N}A\cap [\varphi_N\in U])>0.\eeq

\begin{Remark}\label{simpleREC} In order to give a simple example
of a recurrent cocycle recall that an increasing sequence
$(\ell_{n})$ is called a {\it rigidity sequence} for $T$ if, in the
strong operator topology, $\lim_n T^{\ell_{n}} = I$ where $I$ is the
identity mapping. Suppose that $\varphi:X\to G$ is a cocycle such
that $\varphi_{\ell_n}\to0$ in measure. Then $\varphi$ is recurrent;
indeed, in~(\ref{recu}), $T^{-q_n}A$ is almost equal to $A$ while
$[\varphi_{q_n}\in U]$ is almost the whole space $X$.
\end{Remark}

\begin{Remark} For each $d\geq1$, the cocycle generated by a function
$\varphi: \T \to \R^d$ over any irrational rotation is recurrent if
the components of $\varphi$ have bounded variation and integral~0.
Indeed by Denjoy-Koksma inequality (\ref{f_8}), since
$(\varphi_{q_n})$ is a bounded sequence in $\R^d$ the condition
$\liminf_n \|\varphi_n(x)\| < \infty$ holds for every $x$.

This applies in particular to all piecewise affine or step functions
considered in this paper.\end{Remark}

We always consider recurrent cocycles.

\vskip 3mm A cocycle $\varphi$ is called a {\it coboundary} if
$\varphi=f-f\circ T$ for a measurable map $f : X\rightarrow G$. Two
cocycles $\varphi, \psi$ are called {\it cohomologous} if
$\varphi-\psi$ is a coboundary.

\vskip 3mm {\it Regular cocycles.} An obvious obstruction to the
ergodicity of a cocycle is that $\varphi$ is cohomologous to a
cocycle $\psi$ taking its values in a smaller closed subgroup of
$G$. This suggests the following definition:

\begin{Def} \ A cocycle $\varphi$ is {\it regular} if it is
cohomologous to a cocycle $\psi$ with values in a closed subgroup
$H$ of $G$ such that $T_\psi: (x,h) \to (Tx, h + \psi(x))$ is
ergodic on $X \times H$ for the measure $\mu \times dh$, where $dh$
is the Haar measure on $H$.
\end{Def}

So, a regular cocycle is ``almost'' ergodic (up to reduction by
cohomology to a smaller closed subgroup).

One of the main tools for studying ergodicity and regularity of a
cocycle is the following notion.

{\it Essential value.} An element $g \in \overline{G}$ is called an
{\em essential value} for a cocycle $\varphi$, if for each open
neighborhood $U\ni g$ in $\overline{G}$, for each $A\in\cb$ of
positive measure, there exists $N\in\Z$ such that $\mu(A\cap
T^{-N}A\cap[\varphi_N\in U])>0$. We denote the set of essential
values by $\overline{\E}(\varphi)$ and we set
$\E(\varphi):=\overline{\E}(\varphi)\cap G$.

Note that, if $g\in\E(\varphi)$, we have $\mu(A\cap
T^{-N}A\cap[\varphi_N\in U])>0$ for infinitely many values of
$N\in\Z$. Indeed, because $T$ is ergodic and aperiodic, for each
$N\in\Z\setminus\{0\}$ we can find a subset $C\subset A$, $\mu(C)>0$
such that $T^j C\cap C = \emptyset$, for $|j| \leq N$, $j \not = 0$.
Since $g\in\E(\varphi)$, there is $N_1$ such that $\mu(C\cap
T^{-N_1} C \cap[\varphi_{N_1} \in U])>0$. The property of $C$
implies $|N_1| > |N|$. Iterating this construction, we obtain an
infinite sequence $(N_k)$ such that $\mu(A\cap
T^{-N_k}A\cap[\varphi_{N_k}\in U])>0$.

\begin{Remark}\label{popu} Cocycles with non-trivial essential values must
be recurrent. Indeed, assume that $g\in\E(\varphi)\setminus\{0\}$.
We show Property (\ref{recu}). Take $U$ a neighborhood of $0\in G$.
Then find $N\in\Z$ so that there is $B\subset X$, $\mu(B)>0$ such
that
$$\mbox{$B\subset A$, $T^NB\subset A$ and $\varphi_N(B)\subset g+U$}.$$
Apply once more the definition of the essential value, this time to
the set $T^NB$ to find $C\subset X$, $\mu(C)>0$ and an integer
$M\neq N$ such that
$$\mbox{$C\subset T^NB$, $T^MC\subset T^NB$ and $\varphi_M(C)\subset g+U$.}$$

Now, for $x\in C\subset A$ we have $T^{M-N}x=T^{-N}(T^Mx)\in
T^{-N}(T^NB)=B\subset A$. Moreover, since $T^{M-N}x\in B$,
$$\varphi_{M-N}(x)=\varphi_M(x)+\varphi_{-N}(T^Mx)=\varphi_M(x)-\varphi_N(T^{M-N}x)\in U-U.$$
\end{Remark}

It turns out that $\E(\varphi)$ is a closed subgroup of $G$.
Besides, two cohomologous cocycles have the same group of essential
values.

Let $\sigma_g(x,h) := (x, g+h)$, $g\in G$, be the action of $G$ on
$X\times G$ by translations on the second coordinate. Clearly, it
commutes with $T_{\varphi}$. Then (see \cite{Sc77}, Theorem 5.2)
$\E(\varphi)$ is the {\it stabilizer of the Mackey action} of
$\varphi$, that is
\begin{equation}\label{Mack} \E(\varphi)=\{g\in G:\: F\circ
\sigma_g=F, \forall \, \text{measurable }T_\varphi\text{-invariant}
\, F:X\times G\to \C\}.
\end{equation}
In other words $\E(\varphi)$ is the group of periods of the
measurable $T_\varphi$- invariant functions. Therefore $\varphi$ is
ergodic if and only if $\E(\varphi)=G$. If $\varphi$ is regular,
then the group $H$ in the definition of regularity is necessarily
$\ce(\varphi)$. Coboundaries are precisely regular cocycles
$\varphi$ with $\ce(\varphi)=\{0\}$.

\vskip 3mm The following lemmas show how essential values and
regularity behave when a group homomorphism is applied to a cocycle.

\begin{Lemma} \label{hom} Assume that $\varphi:X\to G$ is a cocycle and let
 $M:G\to H$ be a (continuous) group homomorphism. Then
$M\E(\varphi) \subset \E(M \varphi)$. If $M$ is an isomorphism, then
$M\E(\varphi) = \E(M \varphi)$.
\end{Lemma}
\begin{proof} Let $p \in \E(\varphi)$. We want to show that $Mp$ is a period
of the measurable $T_{M\varphi}$-invariant functions on $X \times
H$. Let $F:X\times H\to\C$ be such a function. Moreover, by a
standard argument, we can modify $F$ on a set of zero measure in
order to obtain a function (still denoted by $F$) which is
$T_{M\varphi}$-invariant everywhere.

Let us fix $h \in H$ and denote $F_h:X\times G \to\C$ by setting
$F_h(x,y)=F(x,h + My)$. We have
\begin{align*}
(F_h\circ T_\varphi)(x,y)&=F_h(Tx,y + \varphi(x))=F(Tx,h + My + M\varphi(x))\\
&=F(x,h + My)=F_h(x,y).
\end{align*}
In view of (\ref{Mack}), $p \in \E(\varphi)$ is a period for $F_h$,
i.e., $F_h(x,y + p) = F_h(x,y)$ for a.e. $(x,y)$. This implies that,
for every $h \in H$ and for a.e. $(x, y)$, $F(x, h + My + Mp) = F(x,
h + My)$.

By Fubini, this implies that there is $y \in G$ such that for a.e.
$(x, h)$, $F(x, h + My + Mp) = F(x, h + My)$. By invariance of the
Haar measure, this implies $F(x, h + Mp) = F(x,h)$, for a.e. $(x,
y)$ and $Mp$ is a period of $F$.

For the second part of the assertion, apply the above to $M\varphi$
and $M^{-1}$.
\end{proof}

\vskip 3mm We have the following lemma (cf. Lemma 2.9 in
\cite{CoFr11}):
\begin{Lemma} \label{quotient} If $\varphi$ is a cocycle on $(X, \mu, T)$
with values in an Abelian l.c.s.c.\ group $G$ and $H$ a closed
subgroup of $G$, then the subgroup $\E(\varphi)/H$ of $G/H$ is such
that
\begin{equation}\label{b} \E(\varphi)/H\subset \E(\varphi+H).
\end{equation}
If $H \subset \E(\varphi)$, then we have \beq\label{aaaa}
\E(\varphi+H) = \E(\varphi)/ H. \eeq Moreover, $\varphi^* :=
\varphi+H:X\to G/H$ is regular if and only if $\varphi$ is regular.
\end{Lemma}
\begin{proof} Whenever $H\subset G$ is a closed subgroup, (\ref{b})
follows from Lemma~\ref{hom} applied to the homomorphism $g \in G
\to g+H \in G/H$.

Now suppose that $H \subset \E(\varphi)$. In view of (\ref{b}) it
remains to show that $\E(\varphi+H)\subset \E(\varphi)/H$. Take
$g_0+H\in \E(\varphi+H)$. All we need to show is that there exists
$h_0\in H$ such that $g_0 +h_0\in \E(\varphi)$, which, by $H\subset
\E(\varphi)$, is equivalent to showing that $g_0\in \E(\varphi)$.

Take $F:X\times G\to\C$ which is measurable and
$T_\varphi$-invariant. Since $H\subset \E(\varphi)$, $F \circ
\sigma_h = F$ for each $h\in H$ because of~(\ref{Mack}). We can
defined $\tilde F$ on $X \times G/H$ such that $\tilde F(x,g+H) =
F(x, g)$. Since $g_0+H\in \E(\varphi+H)$, again using~(\ref{Mack}),
we obtain that $\tilde F\circ\sigma_{g_0+H}=\tilde F$, which by
$H$-invariance of $F$ means $F\circ\sigma_{g_0}=F$ and therefore
$g_0\in \E(\varphi)$.

Assume now that $\varphi^\ast$ is regular. So there are a measurable
$\eta^\ast:X\to G/H$ and a closed subgroup $J^\ast\subset G/H$ such
that
$$ \psi^\ast(x):=\varphi^\ast(x)+\eta^\ast(x)-\eta^\ast(Tx)\in
J^\ast\subset G/H$$ and $T_{\psi^\ast}$ is ergodic on $X\times
J^\ast$, i.e.\ $\ce(\psi^\ast)=J^\ast$. Let $\pi:G\to G/H$ be the
canonical homomorphism and $s:G/H\to G$ a measurable selector, that
is, $s(g+H)\in g+H$ for each $g+H\in G/H$. Then
$J:=\pi^{-1}(J^\ast)$ is a closed subgroup of $G$. Denote
$\eta:=s\circ \eta^\ast$ and set
$$
\varphi'(x):=\varphi(x)+\eta(x)-\eta(Tx).$$ Then
$\varphi'(x)+H=\varphi^\ast(x)+\eta^\ast(x)-\eta^\ast(Tx)=\psi^\ast(x)\in
J^\ast$, whence $\varphi':X\to J$. By~(\ref{aaaa}), since
$\ce(\varphi')=\ce(\varphi)$, we have
$$\ce(\varphi')/H=\ce(\varphi)/H =\ce(\varphi+H)=\ce(\varphi^\ast)=J^\ast,$$
so $\ce(\varphi')=J$ and $\varphi$ is regular.

Conversely, if $\varphi$ is regular then $\varphi=\eta-\eta\circ
T+\psi$, where $\eta:X\to G$ is measurable and
$\psi:X\to\ce(\varphi)$. Then $\varphi^\ast$ is cohomologous to
$\psi+H$ which takes values in $\E(\psi)/H=\ce(\varphi)/H =
\ce(\varphi+H)$ by~(\ref{aaaa}), so $\varphi^\ast$ is regular.
\end{proof}

A particular case is when $H= \E(\varphi)$. For $\varphi^* = \varphi
+ \E(\varphi)$, we get: $\E(\varphi^*) = \{0\}$ and $\varphi$ is
regular if and only if $\varphi^\ast$ is regular (hence a
coboundary).

It can be shown that a cocycle $\varphi$ is a coboundary if and only
if $\overline{\E} (\varphi)=\{0\}$. This includes in particular the
fact that, if $\varphi$ has its values in a compact group and has no
non trivial essential values, it is a coboundary.

{\it Hence regularity is equivalent to $\overline{\E}
(\varphi^{\ast}) = \{0\}$. In particular cocycles with values in
compact groups, or more generally such that $\E(\varphi)$ has a
compact quotient in $G$, are regular.}

\vskip 3mm
\begin{Lemma} \label{nonregIm} Assume that $\varphi:X\to G$ is a cocycle and let
$M:G\to H$ be a (continuous) group homomorphism. If $\varphi:X\to G$
is regular, so is $M\varphi:X\to H$.
\end{Lemma}
\begin{proof}
If $\varphi$ is regular, there is a cocycle $\psi:X\to J$ with
values in a closed subgroup $J\subset G$ and a measurable function
$f:X\to G$ such that $\varphi = f - f \circ T + \psi$ and $T_{\psi}
: (x, j) \to (Tx, j + \psi(x))$ is ergodic on $X \times J$. Thus
$M\varphi = Mf - (Mf) \circ T + M\psi$.

We have $\E(\psi) = J$ by ergodicity of $T_{\psi}$ on $X \times J$
and $MJ = M\E(\psi) \subset \E(M \psi)$ by Lemma~\ref{hom}. Since
$M\psi:X\to MJ\subset\overline{MJ}$, it implies
$\ce(M\psi)\subset\overline{MJ}$. But $\ce(M\psi)$ includes $MJ$ and
is closed, so it is equal to $\overline{MJ}$.

Hence $T_{M\psi}$ is ergodic on $X \times \overline{MJ}$, which
implies the regularity of $M\varphi$.
\end{proof}

The lemma gives a variant of the proof of the second part of Lemma
\ref{quotient}. It shows that if $\varphi$ has a non regular
quotient then it is non regular.

\begin{Remark} \label{dodana} Assume that $\psi: X\to G_1\times G_2$
is a cocycle of the form $\psi=(0,\psi_2)$ with $\psi_2:X\to G_2$.
Then $\E(\psi)=\{0\}\times \E(\psi_2)$. Indeed, $\psi_N(x)$ is close
to $(g_1,g_2)$ if and only if $g_1$ is close to zero and
$(\psi_2)_N(x)$ is close to $g_2$, so this equality follows directly
from the definition of essential value. Moreover, clearly $\psi$ is
a regular cocycle if $\psi_2$ is regular and the converse follows
from Lemma \ref{nonregIm}.
\end{Remark}

 Finally we recall some effective tools which can be used to find
essential values of a cocycle. Given $T:\xbm\rightarrow\xbm$ and
$\varphi:X\rightarrow G$, we denote the image of $\mu$ on $G$ via
$\varphi$ by $\varphi_{\ast}\mu$. We will make use of the following
essential value criterion.

\begin{Prop}[\cite{LePaVo96}]\label{supp}
Assume that $T$ is ergodic and let $\varphi:X\to G$ be a cocycle
with values in an Abelian l.c.s.c.\ group G. Let $(\ell_{n})$ be a
rigidity sequence for $T$. If $(\varphi_{\ell_{n}})_{\ast}\mu\to
\nu$ weakly on $\overline{G}$, then ${\rm supp}(\nu)\subset
\overline{\E}(\varphi)$.
\end{Prop}

Let us recall that all Abelian l.c.s.c.\ groups are metrizable. Let
$d$ be a metric.

\begin{Def} \label{def-period} {\rm We say that $g \in G$ is a {\it quasi-period}
of a cocycle $\varphi$ over $T$ with values in $G$, if there exist
$\delta > 0$, a rigidity sequence $(\ell_{n})$ for $T$ and a
sequence $0<\varepsilon_n \rightarrow 0$ such that
$$\mu(A_n) \ge \delta, \forall n \ge 1, {\rm \ where \ }
A_n = \{x\in X: d(\varphi_{\ell_n}(x) , g) < \varepsilon_n \}.$$
}\end{Def}

\begin{Lemma}\label{lem-period} The set of quasi-periods is included in $\E(\varphi)$.
\end{Lemma}
\begin{proof} \ With no loss of generality we can assume that
$(\varphi_{\ell_{n}})_{\ast}\mu\to\nu$ where $\nu$ is a probability
measure on $\overline{G}$. In view of Proposition \ref{supp} it
suffices to show that a quasi-period $g$ is in the topological
support of $\nu$. Take $U$ a neighborhood of $g$, and select a
smaller neighborhood $g\in V\subset U$ so that $\overline{V}\subset
U$. We have $\nu(U)\geq
\limsup(\varphi_{\ell_{n}})_{\ast}(\mu)(V)=\limsup
\mu(\varphi_{\ell_{n}}^{-1}(V))\geq \limsup \mu(A_{n})\geq \delta$.
\end{proof}

The following ``lifting essential values'' lemma can be applied when
$T$ is an irrational rotation by $\alpha$, $\varphi$ below is
$\R$-valued, centered and of bounded variation (see~(\ref{f_8})),
dealing with different subsequences of the sequence $(q_n)$ of
denominators of $\alpha$.

\begin{Lemma} \label{quasi} Assume that $T$ is ergodic and let
$(\ell_n)$ be a rigidity sequence of $T$. Assume that $\varphi:X\to
H$ is a cocycle such that there exists a compact neighborhood
$C\subset H$ of $0\in H$ for which $\varphi_{\ell_n}\in C$
eventually. Let $\psi:X\to G$ be a cocycle such that
$(\psi_{\ell_n})_\ast(\mu)\to \kappa$ with $\kappa$ a probability
measure on $\overline{G}$. Assume that \beq\label{produ} 0\neq
g_0\in {\rm supp}(\kappa)\cap G.\eeq Then there exists $h_0\in H$
such that $(h_0,g_0)\in \ce(\Phi)$, where $\Phi:=(\varphi,\psi):X\to
H\times G$.
\end{Lemma}

\begin{proof}
Note first that in view of Proposition~\ref{supp},
$g_0\in\ce(\psi)$. By passing to a subsequence if necessary, we can
assume that the distributions of $\varphi_{\ell_n}$ and
$\Phi_{\ell_n}$ converge, that is
$$(\varphi_{\ell_n})_\ast(\mu)\to\nu,\;\;(\Phi_{\ell_n})_\ast(\mu)\to\rho,$$
where $\nu$ is a probability measure on $\overline{H}$, but in fact
(by our standing assumption) which is concentrated on $C$, whence
$\rho$ is a probability measure concentrated on
$C\times\overline{G}$. Moreover, \beq\label{produ1} \mbox{the
projections of $\rho$ on $C$ and $\overline{G}$ are equal to $\nu$
and $\kappa$ respectively.}\eeq

Using~(\ref{produ}), for each $n\geq 1$ select an open neighborhood
$G\supset V_n\ni g_0$ so that $\overline{V}_n$ is compact, ${\rm
diam}\,\overline{V}_n<1/n$, $\kappa(V_n)>0$ and $V_{n+1}\subset
V_n$. In view of~(\ref{produ1}), $\rho(C\times\overline{V}_n)>0$.
Since $C\times \overline{V}_n$ is compact, there is $(c_n,g_n)\in
C\times\overline{V}_n$ such that $(c_n,g_n)\in{\rm supp}(\rho)$ (if
no such a point exists, each point of $C\times \overline{V}_n$ has a
neighborhood which is of measure $\rho$ zero, a finite union of such
neighborhoods must then cover the set $C\times \overline{V}_n$, a
contradiction).

In this way we obtain a sequence $(c_n,g_n)$, $n\geq1$, of points
which are in ${\rm supp}(\rho)\cap C\times\overline{V}_1$ and from
which we can choose a converging subsequence $(c_{n_k},g_{n_k})$.
Moreover, by our assumption on the diameters of $V_n$,
$(c_{n_k},g_{n_k})\to (c,g_0)$, so the result follows.
\end{proof}

In particular, by the proof of Lemma~\ref{lem-period},
Lemma~\ref{quasi} will apply when $g_0 \in G$ is an essential value
of $\psi$ obtained as a quasi-period along a subsequence of the
sequence $(q_n)$ of denominators of $\alpha$.

\subsection{Essential values of cocycles taking values
in $\R^{d}$}

In the lemmas of this subsection, $\Phi$ will stand for a cocycle
with values in $\R^d$.

\begin{Lemma} \label{changeBase}
Let $\theta = (\theta_1, ..., \theta_d)\in\R^d$ be a non zero
essential value of $\Phi$. Then there is a change of basis in $\R^d$
given by a matrix $M$ such that the vector $(1, 0, ..., 0)$ is an
essential value of the cocycle $M\Phi$. If $\theta$ is rational,
then $M$ can be taken rational.
\end{Lemma}
\begin{proof} There is a change of basis in $\R^d$ with $\theta$
as the first vector of the new basis. This can be done via a matrix
$M_1$ with rational coefficients if $\theta \in \Z^d$. The cocycle
$\Phi'= M_1 \Phi$ has an essential value of the form $(\theta_1, 0,
..., 0)$, where $\theta_1$ is a positive real (a positive integer if
$\theta$ is in $\Z^d$, for an adapted choice of $M_1$). By applying
a linear isomorphism $M_2$ (rational in the $\theta$ rational case)
we get that $\Phi^{\prime\prime}= M_2 M_1 \Phi$ has an essential
value of the form $(1, 0, ..., 0)$. \end{proof}

\begin{Lemma} \label{decomp1} There exist a linear isomorphism $M:\R^d\to\R^d$
and integers $d_0,d_1,d_2\geq0$ such that if we set $H_i=\R^{d_i}$,
$i=0,1,2$, then
$$\R^d=H_0\times H_1\times H_2,\;\; M\Phi=(\psi_0,\psi_1,\psi_2)$$
with $\psi_i:X\to H_i$, $i=0,1,2$, and $\E(M\Phi)=\{0\}\times
H_1\times \Gamma_2$, with $\Gamma_2$ a discrete subgroup of $H_2$
such that $H_2 / \Gamma_2$ is compact. If $\Phi$ is a coboundary,
then $d_1 = d_2= 0$
\end{Lemma}
\begin{proof} The group $\E(\Phi)$ is a closed subgroup of $\R^d$,
hence there are linearly independent vectors $v_1,\ldots,v_{d_1}$,
$w_1,\ldots,w_{d_2}$ in $\R^d$ such that
$$\E(\Phi)=\{s_1v_1+\ldots+s_{d_1}v_{d_1}+t_1w_1+\ldots+t_{d_2}w_{d_2}:\:
s_j\in\R,\;t_k\in\Z\}.$$ Select $y_1,\ldots,y_{d_0}\in\R^d$ so that
together with previously chosen $v_j$ and $w_k$ we obtain a basis of
$\R^d$. Then define a linear isomorphism $M$ of $\R^d$ by setting
$$
M(y_i)=e_i,\;M(v_j)=e_{d_0+j},\;M(w_k)=e_{d_0+d_1+k},$$ where
$e_1,\ldots,e_d$ is the standard basis of $\R^d$. Since
$\E(M\Phi)=M\E(\Phi)$, we obtain $\E(M\Phi)=\{0\}\times H_1\times
\Gamma_2$ as required and $M\Phi=(\psi_0,\psi_1,\psi_2)$.
\end{proof}

\vskip 3mm
\begin{Cor}\label{dimension2} Let us consider the case $d=2$.
Let $\Phi=(\varphi^1,\varphi^2):X\to\R^2$ be a cocycle such that
$\ce(\Phi)\neq\{0\}$. Then \beq\label{dim22}\mbox{$\Phi$ is regular
if and only if $a\varphi^1+b\varphi^2:X\to\R$ is regular for each
$a,b\in\R$.}\eeq
\end{Cor}
\begin{proof} In view of Lemma~\ref{nonregIm} we only need to prove
sufficiency. Suppose $\Phi$ is not regular. In view
of~Lemma~\ref{decomp1} we obtain a linear isomorphism
$M:\R^2\to\R^2$ such that $M\Phi=(\psi^0,\psi^i)$ with $\psi^0:X\to
H_0$, $\psi^i:X\to H_i$, $i$ equals either $1$ or $2$ and
$H_0\neq\{0\}$ by non-regularity of $\Phi$ and $H_i\neq\{0\}$ since
$\ce(\Phi)\neq\{0\}$ by hypothesis. Hence $\ce(\psi^0)=\{0\}$ and
there are $a$ and $b$ such that $\psi^0=a\varphi^1+b\varphi^2$. But
$a\varphi^1+b\varphi^2$ is, by assumption, regular, so $\psi^0$ must
be a coboundary. Hence $(\psi^0, \psi^i)$ is cohomologous to $(0,
\psi^i)$ and it now follows from Remark~\ref{dodana} that
$(\psi^0,\psi^i)$ is regular, a contradiction.
\end{proof}

\begin{Lemma} \label{decomp2}  Let $\Phi:X\to\R^d$ be a
recurrent cocycle and let $M:\R^d\to\R^d$ be a linear isomorphism of
$\R^d$ yielding the assertions of the previous lemma. Assume
additionally that the quotient cocycle $\Phi/\E(\Phi)$ is constant.
Then $\psi_0=0$. Moreover, $\Phi$ is regular.
\end{Lemma}
\begin{proof} Since $\E(M\Phi)=M \E(\Phi)=\{0\}\times H_1\times \Gamma_2$,
we have
$$(\psi_0(x),\psi_1(x),\psi_2(x))/\{0\}\times H_1\times
\Gamma_2=const.$$ It follows that there is a constant $b\in\R^{d_0}$
such that $\psi_0=b$. However, $M\Phi$ is recurrent as $\Phi$ is
recurrent and therefore $\psi_0$ is also recurrent. It follows that
$b=0$. Now regularity follows from Remark~\ref{dodana} since
$H_1\times \Gamma_2$ has a compact quotient in $H_1\times H_2$.
\end{proof}

An example of a situation described by the previous lemma is the
following: let $\psi$ be an ergodic step cocycle with values in $\Z$
over an irrational rotation by $\alpha \in (0,1)$. If we modify
$\psi$ by $1_{[0, \alpha)} - \alpha$ which is a coboundary, then for
$\varphi := \psi + 1_{[0, \alpha)} - \alpha$ we have $\E(\varphi) =
\E(\psi) =\Z$; here $\Gamma_2 = \Z$ and $\varphi \text{ mod }
\E(\varphi) = -\alpha$.

\section{Step cocycles over an irrational rotation}\label{nowa}

In this section, we study the regularity of a step $\R^d$-valued
cocycle $\Phi=(\varphi^1,\ldots,\varphi^d)$ over an irrational
rotation $T: x \to x + \alpha$. For such a cocycle the coordinate
$\R$-valued cocycles $\varphi^j$ are integrable and we will
constantly assume that $\int_0^1\varphi^j\,d\mu=0$ with $\mu=m_{\T}$
the Lebesgue measure on $\T^1$, for $j=1,\ldots,d$.

\vskip 3mm {\it Representations of step cocycles}

The coordinates of $\Phi=(\varphi^1,\ldots,\varphi^d)$ can be
(uniquely) represented as follows:
\begin{equation}
\varphi^j(x) = \sum_i t_{i,j} \,(1_{I_{i,j}}(x) - \mu(I_{i,j})),
\label{formCocyGen0}
\end{equation}
where, for $j=1,\ldots,d$, $\{I_{i,j}\}$ is a finite family of
disjoint intervals of $[0,1)$ (covering $[0,1)$ and maximal on which
$\varphi^j$ is constant) and $t_{i,j}\in\R$. Clearly, when $d\geq1$
is fixed, the family of step cocycles forms a linear space over
$\R$.

Setting $\beta_{i,j} = \mu(I_{i,j})$ and $\psi^{i,j}= 1_{I_{i,j}} -
\beta_{i,j}$, we have $\psi^{i,j}_n(x) = \sum_{k=0}^{n-1}
1_{I_{i,j}}(x+k \alpha) - n \beta_{i,j}$; hence the cocycle
$\varphi^j_n$ can be written in the following form:
\begin{equation}
\varphi^j_n(x) = \sum_i t_{i,j} \, \psi^{i,j}_n(x) = \sum_i t_{i,j}
\,(u^{i,j}_{(n)}(x) - \{n \beta_{i,j}\}), \label{formCocyGen}
\end{equation}
with the notation (which is not a cocycle expression)
\begin{equation}\label{11a} u^{i,j}_{(n)}(x) := \psi^{i,j}_n(x)
+\{n\beta_{i,j}\}=\sum_{k=0}^{n-1} 1_{I_{i,j}}(x+k
\alpha)-[n\beta_{i,j}]\in \Z.\end{equation}

\begin{Remark} \label{notAlpha} {\it Without loss of generality, we can assume
that the difference between any two discontinuity points of the
cocycle $\Phi$ is never a multiple of $\alpha$ (modulo~1).} Indeed,
if $\beta$ and $\beta'$ are two discontinuity points of a component
of $\Phi$ such that $\beta' - \beta \in \Z \alpha + \Z$, we can
suppress one of them by adding to $\Phi$ a coboundary, without
changing the ergodic properties of $\Phi$ (we use the fact that
$1_{[\beta, \beta')}(x) - (\beta' - \beta)$ is
coboundary\footnote{Indeed, we have
$1_{[1-\alpha,1)}(x)-\alpha=j(x)-j(x+\alpha)$ with $j(x)=\{x\}$,
then for integers $k,s$
\begin{eqnarray*} &&1_{[1-\{k\alpha+s\}, 1)}(x)-
\{k\alpha+s\}=1_{[1-\{k\alpha\}, 1)}(x)- \{k\alpha\}
=j(x)-j(x+k\alpha)=j_k(x)-j_k(x+\alpha). \end{eqnarray*} The general
case is obtained using the obvious fact that other rotations commute
with $Tx=x+\alpha$.}). In particular, after modification, the
lengths $\mu(I_{i,j})$ in the representation of the new ccocycle are
not in $\Z\alpha + \Z$.
\end{Remark}

{\it Rational step cocycles} \ Assume that $\varphi:\T\to\R$ is a
zero mean step cocycle with its unique
representation~(\ref{formCocyGen0}) of the form
\begin{equation}\label{rat0}
\varphi=\sum_{i=1}^mt_i(1_{I_i}-\mu(I_i)).\end{equation} We say that
$\varphi$ is {\em rational} if there are $c_i\in\Q$, $i=1,...,m$ and
$\beta\in\R$ such that
\begin{equation}\label{rat1}
\varphi=\sum_{i=1}^m c_i1_{I_i}-\beta.\end{equation}

\begin{Lemma}\label{rat2} Assume that $\varphi:\T\to\R$ is a
(zero-mean) step cocycle. The following conditions are equivalent:

(i) $\varphi$ is rational.

(ii) There exists $w\in\R$ such that in the unique
representation~(\ref{rat0}) of $\varphi$ we have $t_i\in w+\Q$ for
$i=1,...,m$.

(iii) $\varphi$ takes values in a coset of $\Q$.

In particular, the family of rational cocycles is a linear space
over $\Q$.
\end{Lemma}
\begin{proof} (i)$\Rightarrow$(ii) \ By (\ref{rat0}),
$\varphi=\sum_{i=1}^mt_i1_{I_i} -\gamma$, where
$\gamma=\sum_{i=1}^mt_i\mu(I_i)$. For $x\in I_i$ we have
$$
c_i-\beta=\varphi(x)=t_i-\gamma,$$ so $t_i\in (\gamma-\beta)+\Q$ for
$i=1,...,m$.

(ii)$\Rightarrow$(iii) For some $r_i\in\Q$, $i=1,...,m$ and $x\in
[0,1)$ we have
$$\varphi(x)=\sum_{i=1}^m(w+r_i)(1_{I_i}(x)-\mu(I_i))=\sum_{i=1}^m
r_i1_{I_i}(x)+(w-\gamma)\in (w-\gamma)+\Q.$$ (iii)$\Rightarrow$(i) \
Take the unique representation~(\ref{rat0}) of $\varphi$:
$\varphi=\sum_{i=1}^mt_i1_{I_i}-\gamma$ with
$\gamma=\sum_{i=1}^mt_i\mu(I_i)$. By assumption, there exists
$\eta\in\R$ such that $\varphi(x)\in\eta+\Q$ for each $x\in[0,1)$.
Thus, for $x\in I_i$ we have $$t_i-\gamma=\varphi(x)=\eta+r_i$$ for
some $r_i\in\Q$. Whence $t_i\in (\gamma+\eta)+\Q$ for $i=1,...,m$.

The latter assertion follows directly from~(iii).
\end{proof}

Suppose that $\varphi$ is rational with a
representation~(\ref{rat1}) and let $\varphi=\sum_{i=1}^m
c_i'1_{I_i}-\beta'$ (with $c_i'\in\Q)$ be another rational
representation. Then by~(iii) of Lemma~\ref{rat2} it follows that
$\beta-\beta'\in\Q$, in other words, in the rational
representation~(\ref{rat1}) the coset $\beta+\Q\in\R/\Q$ is unique.
By $\beta(\varphi)$ we will denote that coset (in fact, less
formally it will be the number $\beta$ in~(\ref{rat1}) understood
modulo $\Q$). Note that $$\varphi(x)\in \beta(\varphi)\;\;\mbox{for
all}\;\;x\in\T.$$ With this in mind we have immediately the
following observation:

\begin{Lemma}\label{rat3} Assume that
$\varphi^1,...,\varphi^d:\T\to\R$ are rational step cocycles. Assume
moreover that $a_j\in\Q$ for $j=1,...,d$ and set
$\varphi=\sum_{j=1}^da_j\varphi^j$. Then
$$
\beta(\varphi)=\sum_{j=1}^da_j\beta(\varphi^j).$$
\end{Lemma}

Now, let $d\geq1$. We say that a step cocycle $\Phi:\T\to\R^d$ is a
{\it rational step cocycle} if its coordinates $\varphi^j$ are
rational, i.e.:
\begin{equation}
\varphi^j =\sum_i c_{i,j} 1_{I_{i,j}} - \beta_j, \label{comblin}
\end{equation}
where the coefficients $c_{i,j}$ are rational numbers and $\beta_j$
is such that $\int_0^1 \varphi^j d\mu = 0$, $j=1,...,d$.

In this case, by replacing $\Phi$ by its non-zero integer multiple
so that all $c_{i,j}$ are integers (recall that a non-zero multiple
of a cocycle $\Phi$ shares its ergodic properties with $\Phi$) we
obtain:
\begin{equation}
\varphi^j_n(x) = u^j_{(n)}(x) - \{n \beta_j\}, \ n \ge 1,
\label{formCocy}
\end{equation}
where the functions $u^j_{(n)}$ have values in $\Z$.

Below we write $\beta_j =\beta_j(\varphi^j)= \beta_j(\Phi)$ (in the
representation~(\ref{comblin})) to stress the dependence of the
$\beta_j$'s on the cocycle $\Phi$. The number of discontinuities of
$\Phi$ is denoted $D(\Phi)$.

We denote by ${\cal L}(\beta_j)$ the set of limit values of the
sequence $(\|q_n \beta_j\| \,)_{n \ge 1}$. Observe that if ${\cal
L}(\beta_j) \not = \{0\}$, there exists a sequence $(n_k)$ such that
$\lim_k \{q_{n_k} \beta_j\} \in (0,1)$.

Let $L:= \max_{i,j} V(\psi^{i,j})$ in case (\ref{formCocyGen}), or
$L:= \max_j V(\varphi^j)$ in case (\ref{formCocy}), where $V$ is the
variation.

${\cal F}$ will denote the interval of integers
\begin{eqnarray} {\cal F}
=\{\ell \in \Z: |\ell| \le L+ 1 \}. \label{valF}
\end{eqnarray}

 From (\ref{f_8}), (\ref{formCocy}) and~(\ref{11a}), it follows
that:
\begin{eqnarray}
u^j_{(q_n)} (x) \in {\cal F}, \ u^{i,j}_{(q_n)}(x) \in {\cal F}.
\label{valF2}
\end{eqnarray}

\vskip 3mm
\subsection{Rational step cocycles} \label {RatStep}

\begin{Lemma} \label{integEssVal}
Let $\Phi$ be a rational step cocycle. If ${\cal L}(\beta_{j_0})
\not = \{0\}$ for some $j_0$, then $\E(\Phi)$ contains a rational
vector $\theta =(\theta_1, ..., \theta_d)$ with $\theta_{j_0} \not =
0$.
\end{Lemma}
\begin{proof} By multiplying $\Phi$ by an integer if needed, we can use
(\ref{formCocy}) with $u^j_{(n)}(x) \in \Z$. We can select a
subsequence $(n_k)$ so that $(\{q_{n_k} \beta_{j}\})_{k \geq 1}$
converges for all $j=1,\ldots,d$ to a limit denoted $\delta_j$, with
$\delta_{j_0} \in (0,1)$. Taking into account~(\ref{valF2}), denote
for $(\ell_1,\ldots, \ell_d) \in{\cal F}^d$
$$A_{k, \ell_1,\ldots, \ell_d}=\{x\in\T:u^j_{(q_{n_k})}(x)=\ell_j,\;j=1, \ldots, d\}.$$
Note that, for each $k\geq1$, $\{A_{k, \ell_1,\ldots,
\ell_d}:\:(\ell_1,\ldots, \ell_d)\in{\cal F}^d\}$ is a partition of
$\T$. By passing to a further subsequence if necessary, we can
assume that $\mu\left(A_{k, \ell_1,\ldots, \ell_d}\right) \to
\gamma_{\ell_1,\ldots, \ell_d} \;\; \mbox{when}\;\;k\to\infty$, for
each $(\ell_1,\ldots, \ell_d)\in{\cal F}^d$. In view
of~(\ref{valF}),~(\ref{formCocy}) and the fact that $\int
\varphi^j\,d\mu=0$, we have
\begin{eqnarray*} &&\sum_{\ell \in{\cal F}} \, \ell \, \mu
\left(\cup_{\ell_1,\ldots, \ell_{j_0-1}, \ell_{j_0+1}, \ldots,
\ell_d\in{\cal F}} \, A_{k, \ell_1,\ldots,
\ell_{j_0-1}, \ell, \ell_{j_0+1},\ldots, \ell_d}\right) \\
&=&\int_0^1 \, u^{j_0}_{(q_{n_k})}(x)\,dx = \int_0^1
\left(\varphi^{j_0}_{q_{n_k}}(x)+ \{q_{n_k}\beta_{j_0}\}\right)\,dx
= \{q_{n_k}\beta_{j_0}\} \to \delta_{j_0}.
\end{eqnarray*}
It follows that
\begin{equation}\label{dod1}
\sum_{\ell \in{\cal F}} \, \ell \sum_{\ell_1, \ldots, \ell_{j_0-1},
\ell_{j_0+1},\ldots, \ell_d \in{\cal F}} \gamma_{\ell_1, \ldots,
\ell_{j_0-1}, \ell, \ell_{j_0+1},\ldots,
\ell_d}=\delta_{j_0}\end{equation} with $\delta_{j_0}\in(0,1)$.
Hence there are $\underline \ell \neq \underline \ell'$ such that
\begin{eqnarray*}
\sum_{\ell_1, \ldots, \ell_{j_0-1}, \ell_{j_0+1}, \ldots, \ell_d
\in{\cal F}} \gamma_{\ell_1, \ldots, \ell_{j_0-1}, \underline \ell,
\ell_{j_0+1},\ldots, \ell_d}>0, \sum_{\ell_1, \ldots, \ell_{j_0-1},
\ell_{j_0+1},\ldots, \ell_d\in{\cal F}} \gamma_{\ell_1, \ldots,
\ell_{j_0-1}, \underline \ell', \ell_{j_0+1},\ldots, \ell_d}>0.
\end{eqnarray*}
Indeed otherwise, $\sum_{\ell_1, \ldots, \ell_{j_0-1}, \ell_{j_0+1},
\ldots, \ell_d\in{\cal F}} \gamma_{\ell_1, \ldots, \ell_{j_0-1},
\ell_0, \ell_{j_0+1},\ldots, \ell_d}=1$ for some $\ell_0\in{\cal F}$
and the other sums are 0, so that the left hand side of~(\ref{dod1})
is an integer, a contradiction. This implies
$$\gamma_{\ell_1,\ldots, \ell_{j_0-1}, \underline \ell, \ell_{j_0+1}, \ldots, \ell_d}>0,\;
\gamma_{\ell'_1, \ldots, \ell'_{j_0-1}, \underline \ell',
\ell'_{j_0+1},\ldots, \ell'_d}>0,$$ for some $d-1$-uples $(\ell_1,
\ldots \ell_{j_0-1}, \ell_{j_0+1}, \ldots, \ell_d)$ and $(\ell'_1,
\ldots \ell'_{j_0-1}, \ell'_{j_0+1},\ldots, \ell'_d)$.

By Lemma~\ref{lem-period} it follows that
\begin{eqnarray*}
&&(\ell_1-\delta_1,\ldots, \ell_{j_0-1}-\delta_{j_0-1}, \underline
\ell - \delta_{j_0}, \ell_{j_0+1}- \delta_{j_0+1},\ldots, \ell_d-\delta_d)\in \E(\Phi), \\
&&(\ell'_1-\delta_1,\ldots, \ell'_{j_0-1}-\delta_{j_0-1}, \underline
\ell' - \delta_{j_0}, \ell'_{j_0+1}- \delta_{j_0+1},\ldots, \ell'_d
- \delta_d)\in \E(\Phi).
\end{eqnarray*}

Thus $(\ell_1- \ell'_1,\ldots, \underline \ell - \underline
\ell',\ldots, \ell_d-l'_d)\in \E(\Phi)$ with $\underline \ell -
\underline \ell'\neq0$ which completes the proof (for the initial
$\Phi$ we have to divide by an integer and obtain a non zero
essential value with rational coordinates).
\end{proof}

\vskip 3mm
\begin{Th} \label{ratReduc}
Let $\Phi$ be a rational step cocycle with values in $\R^d$. There
are $d(\Phi)$, $0 \leq d(\Phi) \leq d$, and a change of basis of
$\R^d$ given by a rational matrix $M$ such that $M \Phi = ({\hat
\varphi}^1, ...,{\hat \varphi}^{d(\Phi)}, {\hat
\varphi}^{d(\Phi)+1}, ..., {\hat \varphi}^d)$ satisfies:

1) $\E(M \Phi)$ contains the subgroup generated by
$$(1, 0, ..., 0), (0, 1, 0, ..., 0), ..., (\underbrace{0,
0, ..., 1}_{d(\Phi)}, 0, ..., 0),$$

2) the cocycle $\hat \Phi = ({\hat \varphi}^{d(\Phi)+1}, ..., {\hat
\varphi}^d)$ is a rational cocycle like (\ref{comblin}) and
satisfies $\lim_n\|q_n \beta_j(\hat \Phi)\| = 0$ for
$j=d(\Phi)+1,\ldots,d$.
\end{Th}
\begin{proof}
We will apply successively Lemmas \ref{integEssVal},
\ref{changeBase} and~\ref{rat3}. If ${\cal L}(\beta_{j}) = \{0\}$
for all $j=1,\ldots,d$, we put $d(\Phi) = 0$. Suppose not all ${\cal
L}(\beta_{j})$ are equal to $\{0\}$, say ${\cal L}(\beta_{1}) \neq
\{0\}$. Then by Lemma~\ref{integEssVal}, there is a rational vector
$\theta = (\theta_1,\ldots, \theta_d)\in \E(\Phi)$ with
$\theta_1\neq0$.

Take a linear (rational) isomorphism $M_1$ of $\R^d$, so that
$M_1(\theta)=e_1$, where $e_1=(1, 0, ..., 0)$ and consider
$M_1(\Phi)=(\varphi'_1, \ldots, \varphi'_d)$. The step cocycles
$\varphi'_2,\ldots,\varphi'_d$ have their own
representation~(\ref{comblin}) with $\beta_j'$ instead of $\beta_j$.
We now look at ${\cal L}(\beta'_j)$ for $j=2,\ldots,d$. If all these
sets are equal to $\{0\}$, we set $d(\Phi)=1$ and the proof is
finished.

Suppose not all ${\cal L}(\beta'_j)$ for $j=2,\ldots,d$ are equal to
zero, say ${\cal L}(\beta'_2)\neq0$. We apply
Lemma~\ref{integEssVal} to $M_1(\Phi)$ and obtain
$\theta'=(\theta'_1,\theta'_2,\ldots)\in \E(M_1\Phi)$ with
$\theta'_2\neq0$. Note that $e_1$ and $\theta'$ are linearly
independent. Then consider a linear (rational) isomorphism $M_2$ of
$\R^d$ that fixes $e_1$ and sends $\theta'$ to $e_2$ and set
$$M_2(M_1(\Phi))=(\varphi'_1, \varphi^{\prime\prime}_2,\ldots,
\varphi^{\prime\prime}_d)$$ (this cocycle has $e_1$ and $e_2$ as its
essential values).

Again, these new cocycles (except for $\varphi'_1$) have their own
representation~(\ref{comblin}) with $\beta^{\prime\prime}_j$ for
$j=2,\ldots,d$. We now look at ${\cal L}(\beta^{\prime\prime}_j)$
for $j=3,\ldots,d$. If the sets ${\cal L}(\beta^{\prime\prime}_j)$,
$j=3,\ldots,d$, are equal to $\{0\}$ we set $d(\Phi)=2$ and the
proof is complete. If not, say ${\cal
L}(\beta^{\prime\prime}_3)\neq\{0\}$, we obtain a rational vector
$\theta^{\prime\prime}=(\theta^{\prime\prime}_1,
\theta^{\prime\prime}_2,\theta^{\prime\prime}_3,\ldots)\in
\E(M_2(M_1(\Phi)))$ with $\theta^{\prime\prime}_3\neq0$. Then
consider a (rational) linear isomorphism $M_3:\R^d\to\R^d$ fixing
$e_1,e_2$ and sending $\theta^{\prime\prime}_3$ into $e_3$ and pass
to the cocycle $M_3(M_2(M_1(\Phi)))$. We complete the proof in
finitely many steps.\end{proof}

\vskip 3mm \subsection{Reduction in the bounded type case}

If we find $d(\Phi) = d$ in Theorem \ref{ratReduc}, then the group
$\E(\Phi)$ contains a subgroup with compact quotient in $\R^d$ and
hence the cocycle $\Phi$ is regular. This is the situation of the
following theorem:

\begin{Th} \label{boundTypeRed}
Let $\alpha$ be of bounded type. Let $(\beta_1, ..., \beta_d)$ be
such that there is no non trivial rational relation between
$1,\alpha, \beta_1, ..., \beta_d$. Then the cocycle $\Phi =(1_{[0,
\beta_j)} - \beta_j)_{j= 1, ..., d}$ is regular. Every step cocycle
$\varphi$ with discontinuities at $\{0, \beta_1, ..., \beta_d \}$
and dimension $d' \leq d$ is regular.
\end{Th}
\begin{proof} We use the notation of Theorem \ref{ratReduc}.
If $d(\Phi) < d$, then we have $\lim_n\|q_n \beta_j(\hat \Phi)\| =
0$ for $j=d(\Phi)+1,\ldots,d$. As $\alpha$ is of bounded type,
taking into account Lemma~\ref{rat3} and Lemma~\ref{qnbeta0}, we
find a non trivial rational relation between the numbers
$\beta_j(\hat \Phi)$. Since the changes of basis are given by
rational matrices $M$, this gives a non trivial rational relation
between $1, \alpha, \beta_1, ..., \beta_d$, contrary to the
assumption of the theorem. Therefore $d(\Phi)=d$ and the cocycle
$\Phi$ is regular. For the second statement, observe that, if $d'
\leq d$, $\varphi$ is the image of $\Phi$ by a linear map. It is
regular if $\Phi$ is regular by Lemma \ref{nonregIm}.
\end{proof}

\begin{Remark}

a) The previous proof is based on the method of rational cocycle,
but applies even to non rational cocycle.

As an illustration of the result, let us consider the cocycle:
$\varphi = \theta 1_{[0, \beta)} - 1_{[0, \theta \beta)}$, with
$(1,\alpha, \beta, \theta\beta)$ rationally independent, $\theta
\not \in \Q$ and $\beta, \theta\beta \in [0, 1)$. This cocycle is
not rational, but obtained from the cocycle $\Phi= (1_{[0, \beta)} -
\beta, 1_{[0, \theta \beta)} - \theta \beta)$ by the map $(y_1, y_2)
\to \theta y_1 - y_2$. For $\alpha$ of bounded type, $\Phi$ and
therefore also $\varphi$ are regular (Theorem \ref{boundTypeRed}).

b) An example of cocycle for which the previous method does not
solve the question of regularity, even for $\alpha$ of bounded type,
is $\varphi = 1_{[0, \beta)} - 1_{[0, \beta)} (.+\gamma) = 1_{[0,
\beta)} - 1_{[0, \beta+ \gamma)} + 1_{[0, \gamma)}$, obtained from
$\Phi = (1_{[0, \beta)} - \beta, 1_{[0, \beta+ \gamma)} - \beta -
\gamma, 1_{[0, \gamma)} - \gamma)$ by the map $(y_1, y_2, y_2) \to
y_1 - y_2 +y_3$.

(See also Example 2, after Theorem \ref{sepClust}.)

c) We would like to mention that when $\beta=1/2$ and $\alpha$ is of
bounded type the regularity (for each $\gamma\in\T$) has been shown
recently by Zhang \cite{Zh10} using different methods. In fact,
Zhang shows that the cocycle
$\Phi=(1_{[0,1/2)}(\cdot)-1/2,1_{[0,1/2)}(\cdot+\gamma)-1/2)$ is
regular (whenever $\alpha$ is of bounded type).

The regularity of $\Phi$ follows also from Lemma \ref{orbi2} and
Theorem \ref{wsd-erg-thm} below (see Corollary \ref{betawsd}).

d) We will give in Subsection \ref{clust} a different method, based
on ``clusters'' of discontinuity points, which can be applied to the
lower dimensional cocycle: $\varphi = 1_{[0, \beta)} - 1_{[0,
\gamma)} - \beta + \gamma$ when $(1,\alpha, \beta, \gamma)$ are
rationally dependent.

e) In the bounded type case, the reduction given by Theorem
\ref{ratReduc} reduces to a cocycle of the form (\ref{comblin}) such
that $\beta_{j} \in \Z \alpha + \Z$ for all $j=d_1+1,...,d$. We can
even obtain $\beta_{j} = 0$ by using the identity: $\alpha = 1_{[1-
\alpha, 1)}(x) + j(x+\alpha)-j(x)$, for $0< \alpha < 1$, where $j(x)
= \{x\}$.
\end{Remark}

\vskip 6mm
\subsection{The case $\|q_n \beta_j\| \to 0$} \label{qnzero}

If $\|q_n \beta_j\| \to 0$, $\forall j$ and $\beta_j \not \in \Z
\alpha + \Z$ (a situation which can occur only for $\alpha$ not of
bounded type), the previous method of reduction cannot be applied.
Nevertheless, there is a first step reduction, based on another
method.

\begin{Lemma}\label{mesu}
Let $\Phi$ be a step function with $D=D(\Phi)$ points of
discontinuity. We have $\mu(A_{q,\ell}(\Phi)) > 1-2D q \varepsilon$,
with $\varepsilon = \ell \|q \alpha\|$, where
$$A_{q,\ell}(\Phi) := \bigcap_{1\le s \le \ell}
\{x\in\T: \Phi_q(x) = \Phi_q(x+sq\alpha) \}, \ \ell, \ q\ge 1.$$
\end{Lemma}
\begin{proof} \ Let $\Delta$ be the set of discontinuities of
$\Phi$. If $x \not \in A_{q,\ell}(\Phi)$, we can find $s$, $1 \le s
< \ell$, and $j$, $0\le j < q$, such that $\Phi(x+j\alpha) \not =
\Phi(x+j\alpha +sq\alpha)$. This implies that $\Phi$ has a
discontinuity at $\delta$ on the circle between $x+j\alpha$ and
$x+j\alpha +sq\alpha$, and therefore $x$ belongs to the interval
$(\delta- j\alpha - \varepsilon, \delta- j\alpha + \varepsilon)$
because $\sup_{1\le s \le \ell} \|sq\alpha\| \le \ell \|q \alpha\|$.
Now, the complement of $A_{q,\ell}(\Phi)$ is included in the set
$\bigcup_{0\le j < q,\, t \in \Delta} B(t - j \alpha, \varepsilon)$,
whose measure is less than $2 D q \varepsilon$.
\end{proof}

\vskip 3mm
\begin{Prop} \label{qn00}
Let $\Phi =(1_{[0, \beta_j)} - \beta_j, j= 1, ..., d)$. Suppose
$\beta_j \not \in \Z\alpha + \Z$ and $\|q_n \beta_j\| \to 0$,
$\forall j$, then $\E(\Phi)$ contains a non zero vector in $\Z^d$ or
a non discrete subgroup of $\R^d$.
\end{Prop}
\begin{proof} For $n \ge 1$, we can write (cf. \ref{formCocy}): $\varphi^j_n(x)
= \tilde u^j_{(n)}(x) - \varepsilon \|n \beta_j\|$, where
$\varepsilon = \pm 1$ and $\tilde u^j_{(n)}$ is the integer valued
function $$\tilde u^j_{(n)} = u^j_{(n)} \text{ if }  \{n \beta_j\} =
\|n \beta_j\|, \ \ \tilde u^j_{(n)}= u^j_{(n)}+1 \text { if } \{n
\beta_j\} = 1- \|n \beta_j\|.$$

a) If $\mu(\{x\in\T: \tilde u^{j_0}_{(q_n)}(x) = 0\}) \not
\rightarrow 1$ for some $j_0$, the proof is similar to the proof of
Lemma~\ref{integEssVal}:  by passing to a subsequence if necessary
to ensure the convergence of all components $\varphi_{q_n}^j$, we
find that $\Phi$ has a quasi-period $(\rho_1, \rho_2, ..., \rho_d)$,
with $ \rho_{j_0} \not = 0$. It follows that $\E(\Phi)$ contains a
non-zero vector in $\Z^d$.

b) Now, we can assume $\lim_n \mu(\{x\in\T: \tilde u^{j}_{(q_n)}(x)
= 0\}) = 1$, for every $j$.

By Lemma \ref{qnbeta0} there is a sequence $(n_k)$ such that
$\|q_{n_k} \beta_{1}\| > {1\over 4} q_{n_k} \|q_{n_k} \alpha\|$. We
put $L_k = [\eta /\|q_{n_k} \beta_{1} \|], \ k \ge 1$, where $\eta$
is such that $\eta < {1\over 16 D}$.

There is at least one index $j_0$ such that, infinitely often,
$\|q_{n_k} \beta_{j_0}\|$ is the biggest value of the set
$\{\|q_{n_k} \beta_{j}\|, j=1,..., D(\Phi)\}$. Hence for $j_0$ and
an infinite subsequence, still denoted $({n_k})$, we have
$$0 < \|q_{n_k} \beta_j\| \leq \|q_{n_k} \beta_{j_0}\|, \, \forall j.$$
In particular, we have $\|q_{n_k} \beta_{j_0}\| \geq {1\over 4}
q_{n_k} \|q_{n_k} \alpha\|$.

Using the notation and the assertion of Lemma \ref{mesu}, we have
$$\mu(A_{q_{n_k},L_k}(\Phi)) > 1 -2D q_{n_k} L_k \|q_{n_k} \alpha\|
\ge 1-8D\eta \ge {1 \over 2}.$$ Moreover, using the definition of
$A_{q_{n_k},L_k}(\Phi)$, for $x \in A_{q_{n_k},L_k}(\Phi)$ and $\ell
\le L_k$, we have
$$\Phi_{\ell q_{n_k}} (x) = \ell \Phi_{q_{n_k}} (x)
= (\ell \tilde u^{j}_{(q_{n_k)}}(x) -\varepsilon \ell
\|q_{n_k}\beta_{j}\|, j=1,..., d)$$ with $\varepsilon=\pm1$. Let
$\rho \in (0, \eta)$. Put $\ell_k:= [\rho/\|q_{n_k}\beta_{j_0}\|] <
\eta/\|q_{n_k}\beta_{j_0}\| \leq L_k +1$. We have, for $x \in
A_{q_{n_k},L_k}(\Phi)$, outside of a set of measure tending to 0,
$$\varphi^{j_0}_{\ell_k q_{n_k}} (x) = \ell_k
\varphi^{j_0}_{q_{n_k}} (x) =\ell_k \tilde u^{j_0}_{(q_{n_k)}} (x) -
\varepsilon\ell_k \|q_{n_k}\beta_{j_0}\| = \pm
[\rho/\|q_{n_k}\beta_{j_0}\|] \, \|q_{n_k}\beta_{j_0}\| \to
\pm\rho.$$ For the other components $j\neq j_0$, outside of a set of
measure tending to 0, we have on $A_{q_{n_k},L_k}(\Phi)$,
$$\varphi^{j}_{\ell_k q_{n_k}} (x) = \ell_k \varphi^{j}_{q_{n_k}}
(x) =\ell_k \tilde u^{j}_{(q_{n_k})} (x) - \varepsilon\ell_k
\|q_{n_k}\beta_{j}\| = \pm \|q_{n_k}\beta_{j}\| \,
[\rho/\|q_{n_k}\beta_{j_0}\|].$$

The above quantity is bounded, since $\|q_{n_k}\beta_{j}\| \,
[\rho/\|q_{n_k}\beta_{j_0}\|] \leq \, \rho \,
\|q_{n_k}\beta_{j}\|/\|q_{n_k}\beta_{j_0}\| \leq \rho$. Passing to a
subsequence still denoted $(n_k)$ if necessary, we obtain that
outside of a set of measure tending to 0, on $A_{q_{n_k},L_k}(\Phi)$
the sequence $(\Phi_{\ell_k \, q_{n_k}} (x))$ converges to the
vector $(\rho_1, \rho_2, ..., \rho_d)$.

Now, the measure of $A_{q_{n_k},L_k}(\Phi)$ is bounded away from 0
and the sequence $(\ell_k q_{n_k})$ is a rigidity sequence for $T$,
since $\ell_k \leq L_k +1$ and
\begin{eqnarray*}
L_k \|q_{n_k} \alpha \| \le \eta {\|q_{n_k} \alpha\| \over \|q_{n_k}
\beta_1 \| } \le 4 {\eta \over q_{n_k}} \rightarrow 0.
\end{eqnarray*}
It follows that, for an arbitrary $\rho \in (0, \eta)$, $ \E(\Phi)$
contains a vector $(\rho_1, \rho_2, ..., \rho_d) \in \E(\Phi)$, with
$\rho_{j_0} = \rho$. It follows that $\E(\Phi)$ includes a
non-discrete subgroup of $\R^d$.
\end{proof}

\begin{Remark} \label{dim1Ess}
Lemma~\ref{integEssVal} and Proposition~\ref{qn00} show that, in
dimension 1, for $\Phi = 1_{[0, \beta]} - \beta$, if $\beta \not \in
\Z \alpha + \Z$, the group $\E(\Phi)$ contains at least a positive
integer.
\end{Remark}

\subsection{Well separated discontinuities, clusters of
discontinuities} \label{clust}

The previous method was based on diophantine properties of the
values of the integrals for rational cocycle. In this subsection we
present results relying on diophantine properties of the
discontinuities of the cocycle. We give sufficient conditions for
regularity of the cocycle defined by a step function $\Phi:\T \to
\R^d$ with integral 0.

The set of discontinuities of $\Phi_n(x)= \sum_{k=0}^{n-1} \Phi(x +
k\alpha)$ is ${\cal D}_n := \big\{ \{x_i - k\al \}:\,1\le i\le
D,0\le k < n \big\}$. We assume that the points $x_i - k\al {\rm
\mod 1}$, for $1\le i\le D,0\le k < n$, are distinct. The jump of
$\Phi$ at $x_i$ is $\sigma_i = \sigma(x_i) = \Phi(x_i^+) -
\Phi(x_i^-)$. A discontinuity of the form $\{x_i - k\al \}$ is said
to be of type $x_i$.

By Lemma \ref{cont2qn}, any interval of the circle of length $\ge
2/q_n$ contains at least one point of the set $\big\{\{x_i - k
\alpha\}, k = 0, \dots, q_n-1\big\}$, hence at least one
discontinuity of $\Phi_{q_n}$ of type $x_i$ for each $x_i \in {\cal
D}$.

\vskip 3mm \goodbreak {\bf Well separated discontinuities}

\vskip 3mm We write ${\cal D}_n=\{\gamma_{n,1} < ... <\gamma_{n,Dn}
<1\}$ and $\gamma_{n,Dn+1} = \gamma_{n,1}$, where, for $1\le\ell\le
Dn$, the points $\gamma_{n,\ell}$ run through the set of
discontinuities ${\cal D}_n$ in the natural order.

%The following property of well distribution or separation of the
%discontinuities will be used.

\begin{Def} \label{propert_def}
{\rm The cocycle is said to have {\em well separated discontinuities
(wsd)}, if there is $c >0$ and an infinite set ${\cal Q}$ of
denominators of $\alpha$ such that
\begin{eqnarray}
\gamma_{q,\ell+1} - \gamma_{q,\ell} \ge c/q, \ \forall q \in {\cal
Q}, \, \ell \in \{1,\dots,Dq\}. \label{wsdPrty}
\end{eqnarray}
}\end{Def}

This condition is similar to Boshernitzan's condition (\cite{Bo85})
for interval exchange transformations. The result below extends an
analogous statement when $\Phi$ takes values in $\Z^d$ (see
\cite{CoGu12}).

\begin{Th}\label{wsd-erg-thm}
Let $\Phi$ be a zero mean step function. If $\Phi$ satisfies the wsd
property~(\ref{wsdPrty}), then the group $ \E(\Phi)$ includes the
set $\{\sigma_i:\: i=1, \ldots, D\}$ of jumps at discontinuities of
$\Phi$. Moreover, $\Phi$ is regular.\end{Th}

\begin{proof} Let us consider $\Phi_q(x)$ for $q \in {\cal Q}$. By
(\ref{formCocyGen}) and (\ref{valF2}), we can write, with
$u^{i,j}_{(q)}(x)$ in a finite fixed set of integers $\cal F$,
$$\Phi_{q} = (\varphi^j_{q})_{j=1, ..., d} \ {\rm with \ }
 \varphi^j_{q}(x) = \sum_i t_{i,j} \, u^{i,j}_{(q)}(x) - \sum_i
t_{i,j} \,\{{q} \beta_{i,j}\}.$$ Let $\theta^{(q)} := (\theta_{j,q},
j = 1, ..., d)$ with $\theta_{j,q}:= -\sum_i t_{i,j} \,\{{q}
\beta_{i,j}\}$. We can assume that the limit $\theta := \underset{q
\to \infty, \, q \in {\cal Q}} \lim \theta^{(q)}$ exists. The set of
values of $\Phi_q$ for $q \in {\cal Q}$ is included in $R +
\theta^{(q)}$ where $R$ is the finite fixed set of vectors
$\{(\sum_{i} t_{i,j}k_{i,j}, j=1, ..., d):\: k_{i,j} \in \cal F\}$.

Let ${\mathcal I}_q$ be the partition of the circle into the
intervals of continuity of $\Phi_q$, $I_{q,\ell}= [\gamma_{q,\ell},
\gamma_{q,\ell+1})$, $1\le \ell \le Dq$. With the constant $c$
introduced in (\ref{wsdPrty}), let $J_{q, \ell} \subset \T$ be the
union of $L:= \lfloor 2/c\rfloor+1$ consecutive intervals in
${\mathcal I} _q$ starting with $I_{q,\ell}$. By~(\ref{wsdPrty})
every $J_{q,\ell}$ has length $\geq 2/q$, thus contains an element
of the set $\big\{\{x_i - s\alpha\}, s=0,\dots,q-1\big\}$ for each
$x_i$.

Therefore, for every jump $\sigma_i$ of $\Phi$, there is $v\in R$
and two consecutive intervals $I,I'\in{\mathcal I} _q$, with $I\cup
I'\subset J_{q, \ell}$, such that the value of $\Phi_q$ is $v +
\theta^{(q)}$ on $I$ and $v+ \theta^{(q)}+ \sigma_i$ on $I'$.

Given $i\in\{1,...,D\}$, we denote ${\mathcal H}_q(\sigma_i)$ the
family of intervals $I \in{\mathcal I}_q$ such that the jump of
$\Phi_q$ at the right endpoint of $I$ is $\sigma_i$. Since each
interval $J_{q, \ell}$ contains an interval $I \in {\mathcal
H}_q(\sigma_i)$, the cardinality of ${\mathcal H}_q(\sigma_i)$ is at
least ${q D }\over L$.

Fix additionally $v\in R$ and let ${\mathcal A_q}(\sigma_i,v)$ be
the set of intervals $I\in{\mathcal H}_q(\sigma_i)$ such that
$\Phi_q(x) = v + \theta^{(q)}$ on $I$. Let ${\mathcal
A_q'}(\sigma_i,v)$ be the set of intervals $I' \in {\mathcal I}_q$
adjacent on the right to the intervals $I \in {\mathcal
A_q}(\sigma_i,v)$.

Let $A_q(\sigma_i,v)$ be the union of intervals $I \in {\mathcal
A_q}(\sigma_i,v)$ and $A'_q(\sigma_i,v)$ the union of intervals
$I'\in {\mathcal A_q'}(\sigma_i,v)$. The value of $\Phi_q$ is $v +
\theta^{(q)}$ on $A_q(\sigma_i,v)$ and $v+ \theta^{(q)} + \sigma_i$
on $A'_q(\sigma_i,v)$.

There is $v_0 \in R$ and an infinite subset ${\cal Q}_0$ of ${\cal
Q}$ such that, for $q\in {\cal Q}_0$,
\begin{equation} \label{proport_eq}
|{\mathcal A_q}(\sigma_i,v_0)|, \ |{\mathcal A_q'}(\sigma_i,v_0)|\ge
{|{\mathcal H}_q(\sigma_i)| \over |R|} \ge {qD \over L|R|}.
\end{equation}

By (\ref{wsdPrty}) and (\ref{proport_eq}), we have
$\mu\left(A_q(\sigma_i,v_0)\right), \
\mu\left(A'_q(\sigma_i,v_0)\right)\ge {Dc^2 \over (2+c)|R|}$. Thus
$v_0+ \theta$ and $v_0 + \theta +\sigma_i$ are quasi-periods, hence,
by Lemma \ref{lem-period} essential values. Since $\E(\Phi)$ is a
group, $\sigma_i$ is an essential value. Therefore $\E(\Phi)$
includes the group generated by the jumps of $\Phi$.

Finally, notice that the quotient cocycle $\Phi /{\cal E}(\Phi)$ is
a continuous step cocycle, hence is constant. Therefore, the
regularity of $\Phi$ follows from Lemma~\ref{decomp2}.
\end{proof}

For $\Phi_d:= \bigl(1_{[0, \beta_j]} - \beta_j, \,j= 1, ...,
d\bigr)$ with $\beta_i\neq\beta_j$ whenever $i\neq j$, the jump of
$\Phi_d$ is $(1, ..., 1)$ at 0 and $(0,...,0, -1, 0,..., 0)$ at
$\beta_j$ ($-1$ stands at the $j$-th coordinate), $j=1,\ldots,d$. If
the wsd property is satisfied, the group $\E(\Phi)$ includes $\Z^d$.
Therefore the cocycle $\Phi_d$ is regular whenever the wsd property
holds.

In view of Lemma~\ref{orbi2} and Theorem~\ref{wsd-erg-thm}, we
obtain the following result (where the case  $\beta \in \Z\alpha +
\Z$ can be treated directly).

\begin{Cor}\label{betawsd}
Let $\alpha$ be of bounded type. Let $\beta$ be a real number.

1) The cocycle $(1_{[0, {r \over s})}(.) - {r \over s}, \,1_{[0, {r
\over s})}(. +\beta) - {r \over s})$ is regular for every rational
number ${r \over s} \in (0,1)$.

2) If ${r_1 \over s_1}, ..., {r_d \over s_d}$ are rational numbers
such that $0 < {r_i \over s_i} \beta<1$, then, for every real
numbers $t_1, ..., t_d$, the cocycle $\varphi = \sum_i t_i 1_{[0,
{r_i \over s_i} \beta)}- \beta \sum_i t_i {r_i \over s_i}$ is
regular.
\end{Cor}

\vskip 3mm {\bf Clusters of discontinuities}

 For a subset $C$ of discontinuities of $\Phi$, we denote $\sigma(C)
= \sum_{x_i \in C} \sigma(x_i)$ the corresponding sum of jumps of
$\Phi$. The number of discontinuities of $\Phi$ is $D=D(\Phi)$. The
following result can be useful when the discontinuities are not well
separated.
\begin{Th} \label{sepClust} Suppose that there are two discontinuities
$x_{i_0}, x_{j_0}$ of $\Phi$ and a subsequence $(q_{n_k})$ such that
for a constant $\kappa>0$ we have \beq\label{aa1}
q_{n_k}\|(x_{i_0}-x_{j_0}) - r\alpha\| \geq \kappa, \ \forall \
|r|<q_{n_k}. \eeq Then, if the sum $\sigma(C)$ is $\not = 0$ for
each non-empty proper subset $C$ of the set of discontinuities of
$\Phi$, then $\Phi$ has a non trivial essential value.
\end{Th}
\begin{proof} By Lemma \ref{cont2qn} any interval of length $2/q_n$
on the circle contains at least one discontinuity of each type $x_i$
and at most 4 such discontinuities, therefore at most $4D(\Phi)$
discontinuities of $\Phi_{q_n}$.

Consider the sequence ${\cal Q} = (q_{n_k})$ of denominators
satisfying~(\ref{aa1}). On the circle $\T$ we will deal with
families of disjoint intervals of length $4/q_{n_k}$. In fact, we
consider families of the form $\big\{I_j^{(k)}:\: j \in J_k \subset
\{0, 1, ..., q_{n_k}-1\}\big\}$ with $I_j^{(k)} = I_0^{(k)} +
\{-j\alpha\}$, where $I_0^{(k)} = [0, 4/q_{n_k}]$ and $J_k$ is such
that its cardinality satisfies $|J_k| \geq \delta_1 q_{n_k}$ for a
fixed positive constant $\delta_1$.

The number of different ``patterns of discontinuities'' (i.e.\
consecutive types of discontinuities) which can occur altogether in
these intervals is finite (indeed, the length of a pattern of
discontinuity is bounded by $8D(\Phi)$). There are an infinite
subsequence of ${\cal Q}$ (still denoted by ${\cal Q}$) and a family
of intervals $I_0^{(k)} + \{-j\alpha\}$, $j \in J_k'$ with $|J_k'|
\geq \delta_2 q_{n_k}$ for a fixed positive constant $\delta_2$
(therefore with a total amount of measure bounded away from~0) such
that the same pattern of discontinuities occurs in each interval of
the family. For illustration, if the cocycle has 4 discontinuities
$x_1, x_2, x_3, x_4$, we can have for instance in each interval the
pattern $(x_1, x_3, x_4, x_3, x_2, x_1, x_2, x_4)$, corresponding in
a given interval to the ``configuration'' (a sequence of
discontinuities) of the form $(\{x_1 -\ell_{j,1} \alpha\}, \{x_3
-\ell_{j,2} \alpha\}, \{x_4 -\ell_{j,3} \alpha\}, \{x_3 -\ell_{j,4}
\alpha\}, \{x_2 -\ell_{j,5} \alpha\}, \{x_1 -\ell_{j,6} \alpha\},
\{x_2 -\ell_{j,7} \alpha\}, \{x_4 -\ell_{j,8} \alpha\})$.

Now, by taking a further subsequence of ${\cal Q}$ if necessary, we
will assure a convergence at scale $1/q_{n_k}$ for the
discontinuities in $I_j^{(k)}$. More precisely, observe that if
$\{x_i- \ell \alpha\} \in I_j^{(k)}$, then $\{x_i- \ell \alpha\} -
\{- j \alpha\} \in I_0^{(k)}$. Hence $\{x_i- (\ell - j) \alpha\}$ is
in $I_0^{(k)}$ and therefore it belongs to the set $\big\{\{x_i - u
\alpha\}:\: |u| < q_{n_k}\big\} \cap I_0^{(k)}$. Notice that this
set $\big\{\{x_i - u \alpha\}:\: |u| < q_{n_k}\big\} \cap I_0^{(k)}$
has at least 2 elements and has no more that 8 elements and it does
not depend on $j$ (when $k$ changes, the set $J'_k$ does and so $j$
are different for different $k$, however the common shift, namely
the shift by $j\alpha$, leads to points which will be common for all
$j\in J'_k$; on the other hand $r$ runs over a fixed set as the
patterns of discontinuities are the same regardless $k$ and $j$).
Therefore we can write it explicitly as $\big\{\{x_i- u_{n_k,i,r}
\alpha\}\big\}$.

We can extract a new subsequence of ${\cal Q}$ (for which we still
keep the same notation ${\cal Q} = (q_{n_k})$) such that for each
$\{x_i- u_{n_k,i,r} \alpha\}$ the sequence $q_{n_k} \{x_i-
u_{n_k,i,r} \alpha\}$ converges to a limit $y_{i,r}\in[0,4]$ when $k
\to \infty$. This is possible, since there is a finite number of
such points in $I_0^{(k)}$ for each $n_k$.

Therefore the configurations of discontinuities in the intervals
$I_j^{(k)}$ for $j \in J_k'$ are converging at the scale
$1/q_{n_k}$, i.e.\ after applying the affinities $x \to q_{n_k} (x -
\{-j\alpha\})$. We can group the discontinuities (of type) $x_i$
according to the value of the limit $y_{i,r}$.

We call ``clusters'' the subsets of discontinuity points in
$I_j^{(k)}$ with the same limit at the scale $q_{n_k}$ (hence, such
that the corresponding limits $y_{i,r}$ in $[0,4]$ coincide).
Observe that two discontinuities of the same type $x_i$ are at
distance $\geq {1 \over 2 q_{n_k}}$ by the point 4) of
Lemma~\ref{cont2qn} and therefore are not in the same cluster: a
cluster contains at most one discontinuity of a given type $x_i$. In
view of~(\ref{aa1}), the number of elements in a cluster is strictly
less than $D(\Phi)$ the number of discontinuities of $\Phi$.

By passing once more to a subsequence of ${\cal Q}$ (still denoted
by ${\cal Q}=(q_{n_k})$) if necessary, we extract a sequence of
families of disjoint ``good'' intervals of length $4/q_{n_k}$ with
the same configuration of clusters inside the intervals of a family.
There are at least three different clusters in each ``good''
interval (since for an interval of length $4/q_{n_k}$ a given type
of discontinuity occurs at least twice and must occur in different
clusters as shown above, moreover the number of elements in a
cluster is at most $D(\Phi)-1$). The clusters in each interval are
separated by more than $c/q_{n_k}$. As in the proof of Theorem
\ref{wsd-erg-thm}, the values of the cocycle at time $q_{n_k}$ are
$v + \theta^{(q_{n_k})}$ with $v$ in a fixed finite set and
$(\theta^{(q_{n_k})})$ a converging sequence.

For $k$ large, clusters of discontinuities are separated by
intervals of order $c_1/q_{n_k}$ for a fixed positive constant $c_1$
and there are at least 3 clusters in a ``good'' interval
$I^{(k)}_j$. The number of intervals in the families is greater than
a fixed fraction of $q_{n_k}$. It follows that, under the assumption
that the sum of jumps $\sigma(C)$ is $\not = 0$ for each non-empty
proper subset $C$ of the set of discontinuities of $\Phi$, the
cocycle at time $q_{n_k}$ is close to a non zero constant on a set
which has a measure bounded away from 0.

Therefore there $\Phi$ has a non trivial quasi-period, hence a non
trivial finite essential value.
\end{proof}

Recall that, by Remark~\ref{notAlpha}, if $x_1,\ldots,x_D$ are all
discontinuities of a step cocycle $\Phi$, then for $i\neq j$ we can
assume that $x_i-x_j$ is not a multiple of $\alpha$ modulo~1. Assume
that $\alpha$ is of bounded type. Then, fixing $i_0\neq j_0$ and
using Lemma~\ref{orbi2} to select a subsequence $(q_{n_k})$ so that
(\ref{aa1}) holds for a constant $\kappa>0$, the assumption of the
theorem are fulfilled.

\vskip 3mm {\it Example 1: cocycle with 3 discontinuities}

Let $\alpha$ be an irrational number {\it of bounded type}. Let
$\varphi$ be a scalar cocycle with 3 effective discontinuities $0,
\beta, \gamma$. The sum of jumps for the 3 discontinuities is 0, and
for subsets of 1 or of 2 discontinuities it is always non zero. If
$\beta$ (resp. $\gamma$) is not in $\Z\alpha + \Z$, by Lemma
\ref{orbi2} there are subsequences of denominators along which the
discontinuities of type $\beta$ (resp.$\ \gamma$) belong to clusters
which reduce to a single discontinuity or to two discontinuities.
Therefore, by Theorem \ref{sepClust} the group of finite essential
values does not reduce to $\{0\}$.

\vskip 3mm {\it Example 2: cocycle with 4 discontinuities}

Let us consider the $\R$-valued cocycle
$a(1_{[0,\beta)}(\cdot)-\beta) -(1_{[0,\beta)}(\cdot -
\gamma)-\beta)$ with $\beta<\gamma$.

There are 4 discontinuity points: $(0, \beta, \gamma, \beta +
\gamma)$ with respective jumps $+a$, $-a$, $-1$, $+1$. Assume that
$\beta$ is such that there is a subsequence $(q_{n_k})$ and a
constant $\kappa>0$ such that \beq\label{aa1z} q_{n_k}\|\beta -
r\alpha\| \geq \kappa, \ \forall \ |r|<q_{n_k}. \eeq

We apply the method of Theorem \ref{sepClust}, with the subsequence
$(q_{n_k})$. By the above condition on $\beta$, in a cluster we can
find either a single discontinuity, or two discontinuities of type
in $(0, \gamma)$, $(0, \beta + \gamma)$, $(\beta , \gamma) , (\beta,
\beta+\gamma)$ with respective sum of jumps: $a-1$, $a+1$, $-(a+1)$,
$-a +1$. The case of 3 discontinuities is excluded.

Therefore, if $a \not \in \{\pm 1\}$, we have a non trivial
essential value. When $a = -1$, then the cocycle reads
$-1_{[0,\beta)}(\cdot) -1_{[0,\beta)}(\cdot - \gamma) + 2 \beta$,
and by Theorem~\ref{boundTypeRed} or the method of
Proposition~\ref{qn00} we obtain a non trivial essential value.

So for the classification of the cocycle
$a(1_{[0,\beta)}(\cdot)-\beta) -(1_{[0,\beta)}(\cdot -
\gamma)-\beta)$ the only case to be considered is $a=1$. This leaves
open the question of the regularity of the cocycle
$1_{[0,\beta)}(\cdot) -(1_{[0,\beta)}(\cdot - \gamma)$ for $\alpha$
of bounded type and any $\beta, \gamma$.

\subsection{On the regularity of $\Phi_d$, $d=1,2,3$} \label{123}

{\bf d= 1, $\Phi_1= 1_{[0, \beta)} - \beta$}

\begin{Th}\label{case1} The cocycle $\Phi_\beta = 1_{[0,
\beta)} - \beta$ is regular over any irrational rotation.
\end{Th} \begin{proof}
If $\beta \in \Z \alpha + \Z$, then $\Phi_\beta$ is a coboundary
(see Remark~\ref{notAlpha}). Suppose that $\beta \not \in \Z \alpha
+ \Z$. Then, by Lemma~\ref{integEssVal} and Proposition~\ref{qn00},
there is a positive integer in the group $\E(\Phi)$ (cf. Remark
\ref{dim1Ess}). Therefore the cocycle $\Phi_\beta$ is always
regular. \end{proof}

\begin{Remark}
If $\beta, \alpha, 1$ are independent over $\Q$, then by a result of
Oren (\cite{Or83}) the cocycle defined by $\Phi_\beta$ is ergodic.
\end{Remark}

\vskip 3mm {\bf d= 2, $\Phi_2= (1_{[0, \beta)} - \beta, 1_{[0,
\gamma)} - \gamma)$}

\vskip 2mm a) {\bf $\alpha$ of bounded type}

\begin{Th}\label{case11} If $\alpha$ is of bounded type, the cocycle
$\Phi_2= (1_{[0, \beta)} - \beta, 1_{[0, \gamma)} - \gamma)$ is
regular.
\end{Th}
\begin{proof} Recall that we constantly assume that $\beta,\gamma, \beta-\gamma$
are not in $\Z\alpha+\Z$. The proof is done in three steps:

{\it Step 1.} \ $\ce(\Phi_2)\neq\{0\}$; indeed, this follows
immediately from the proof of Theorem~\ref{ratReduc} applied to
$\beta\notin\Z\alpha+\Z$ (see Lemma~\ref{qnbeta0} and
Lemma~\ref{integEssVal}).

{\it Step 2.} \ $\beta=\gamma$; then our cocycle is regular by
Theorem~\ref{case1}.

{\it Step 3.} \ $0<\beta<\gamma<1$. Now, we claim that for each
$a,b\in\R$ the cocycle $a(1_{[0,\beta)}-\beta)+
b(1_{[0,\gamma)}-\gamma)$ is regular. Indeed, we have already
noticed this property to hold if $a$ or $b$ is equal to zero. When
$a\neq0\neq b$, we obtain a step cocycle with 3 effective
discontinuities $0,\beta$ and $\gamma$. In that case we apply
Theorem~\ref{sepClust} (see the application to cocycles with 3
discontinuities, example 1 after the proof) to conclude that our
scalar cocycle has a non-zero finite essential value and hence is
regular. The claim immediately follows. The regularity of $\Phi$ is
now an immediate consequence of Corollary~\ref{dimension2}.
\end{proof}

\begin{Remark} Notice that we can apply other previous results
to obtain another, more complex proof of Theorem~\ref{case11}, which
however can be applied in other situations. Indeed, since $\alpha$
of bounded type, we apply Theorem~\ref{boundTypeRed} to conclude
that the cocycle $\Phi$ is regular whenever $\beta, \gamma, \alpha,
1$ are independent over $\Q$.

Otherwise, there are integers $r, s, v, w$ not all equal to zero
such that $r \beta + s \gamma + v\alpha +w = 0$.

The case when $\beta$ or $\gamma$ belongs to $\Z \alpha + \Z$ is
excluded (cf.\ Remark \ref{notAlpha}).

1) Assume that $\beta, \gamma \not \in \Q \alpha + \Q$ and $\beta -
\gamma\notin\Q\alpha+\Q$.

If $s$ or $r\neq0$, say $s\neq0$ then $\gamma = -{r \over s} \beta -
{v \over s} \alpha -{w \over s}$. We apply Lemma~\ref{orbi2} for
$\beta_1=\frac{1}{1}\beta+\frac{0}{1}\alpha+\frac{0}{1}$,
$\beta_2=\frac{-r}{s}\beta+\frac{-v}{s}\alpha+\frac{-w}{s}$ and
$\beta_3=\frac{-r-s}{s}\beta +\frac{-v}{s}\alpha+\frac{-w}{s}$ and
obtain a subsequence $(q_{n_k})$ along which the wsd property is
satisfied for the discontinuities of $\Phi_2$. Then
Theorem~\ref{wsd-erg-thm} applies.

2) Suppose $s =0$ and $\gamma\notin\Q\alpha+\Q$,
$\beta\in\Q\alpha+\Q$. It is enough to show that $d_1=2$ in
Theorem~\ref{ratReduc}. By the proof of that theorem applied to
$\beta\notin\Z\alpha+\Z$, in view of Lemma~\ref{qnbeta0}, we obtain
$M:\R^2\to\R^2$ a rational change of coordinates such that
$M\Phi_2=(\psi^1,\psi^2)$ has $(1,0)$ as its essential value. On the
other hand, by Lemma~\ref{rat3} (taking into account that ${\rm
det}\,M\neq0$) and remembering that under our assumption $\beta$ and
$\gamma$ are independent over $\Q$, we obtain that
$\beta(\psi^i)\notin\Z\alpha+\Z$, $i=1,2$. Therefore, again by
Lemma~\ref{qnbeta0}, ${\cal L}(\beta(\psi^i))\neq\{0\}$, hence by
the proof of Theorem~\ref{ratReduc}, $d_1=2$.

3) The missing case $\beta- \gamma \in \Q \alpha + \Q$ (see the
assumption in 1) and the separate case $\beta,\gamma\in
\Q\alpha+\Q$) are covered by Lemma \ref{orbi2} and an application of
Theorem \ref{wsd-erg-thm}.
\end{Remark}

\vskip 2mm b) {\bf $\alpha$ of non bounded type}

 For $d =2$ and $\alpha$ not of bounded type the question of construction of a
non regular step function is not solved and the purpose of this
paragraph is to present some observations.

 From Lemma~\ref{integEssVal} and Proposition~\ref{qn00}, we know that
$\E(\Phi)$ does not reduce to $\{0\}$. By
Corollary~\ref{dimension2}, the regularity of the cocycle is
equivalent to the regularity of the one dimensional cocycles with 3
discontinuities: $\varphi = a(1_{[0, \beta)} - \beta) - b(1_{[0,
\gamma)} - \gamma)$, where $a, b$ are arbitrary real numbers. Since
we know already that regularity holds for $b=0$, it suffices to
consider $\varphi = a1_{[0, \beta)} - 1_{[0, \gamma)} -(a\beta -
\gamma)$.

It is interesting to understand the particular case $\gamma = \ell
\beta$, with $\ell$ a positive integer. We will give some partial
results on this cocycle and ask questions.

First of all, there are special situations where one can conclude
that the cocycle $\varphi = \ell 1_{[0, \beta)} - 1_{[0,
\ell\beta)}$ is a coboundary (we assume that $\ell\beta<1$). We use
the following result of Guenais and Parreau (with the notation of
Section~\ref{subsectIrr}, in particular $Tx=x+\alpha$):

\begin{Th}\label{NSCdtion} (\cite{GuPa06}, Theorem 2)
Let $\varphi$ be a step function on $\T$ with integral 0 and jumps
$-s_j$ at distinct points $(\beta_j, 0\leq j\leq m$), $m\geq 1$, and
let $t\in \T$. Suppose that there is a partition $\cal P$ of
$\{0,\ldots,m\}$ such that for every $J\in {\cal P}$ and $\beta_J\in
\{\beta_j:\:j \in J\}$: \hfill \break (i) \ $\sum_{j\in J}s_j \in
\Z$; \hfill \break (ii)\ for every $j \in J$, there is a sequence of
integers $(b_n^j)_n$ such that
$$\beta_j= \beta_{J}+ \sum_{n\geq 0}
b_n^jq_n\alpha \mod 1, {\ with \ } \sum_{n\geq 0}
\frac{|b_n^j|}{a_{n+1}}<+\infty \, \hbox{\ and \ } \, \sum_{n\geq
0}\Bigl\|\sum_{j \in J}b_n^js_j\Bigr\|^2 <+\infty;$$ (iii) \ there
is an integer $k'$ such that $t=k'\alpha -\sum_{J\in {\cal P}}t_J$
where
$$t_J=\beta_J\sum_{j\in J}s_j
+\sum_{n\geq 0}\Bigl[\sum_{j\in J}b_n^js_j\Bigr]q_n\alpha \mod 1.$$

Then there is a measurable function $f$ of modulus 1 solution of
\begin{equation}
e^{2i\pi \varphi}= e^{2i\pi t} f \circ T/f. \label{MultEquat}
\end{equation}

Conversely, when $\sum_{j \in J} s_j \notin \Z$ for every proper non
empty subset $J$ of $\{0,..,m\}$, these conditions are necessary for
the existence of a solution of (\ref{MultEquat}).
\end{Th}

Take $\varphi = \ell 1_{[0, \beta]} - 1_{[0, \ell \beta]}$. With the
previous notation, the discontinuities are at $\beta_0 = 0, \beta_1
= \beta, \beta_2 = \gamma = \ell \beta$ ($m= 2$) with jumps $\ell
-1, -\ell , 1$ respectively and the partition ${\cal P}$ is the
trivial partition with the single atom $J = \{0,1,2\}$. We also have
$\beta_J = 0$, $\sum_{j\in J} s_j = 0$.

Suppose that the parameter $\beta$ has an expansion in base $(q_n
\alpha)$ (Ostrowski expansion, see \cite{IN}):
\begin{eqnarray}
\beta= \sum_{n\geq 0} b_n q_n\alpha \mod 1, {\rm \ with \ }
\sum_{n\geq 0} {|b_n| \over a_{n+1}}<+\infty, \ b_n \in \Z.
\label{betaExpan}
\end{eqnarray}
We can take $b_n^0 = 0, b_n^1 = b_n, b_n^2 = \ell b_n$, so that
$\sum_{j \in J} b_n^j s_j = \ell b_n - \ell b_n = 0$. In view of
Theorem~\ref{NSCdtion}, for every real $s$, the multiplicative
equation $e^{2\pi i s \varphi} = f\circ T/f$ has a measurable
solution $f:\T\to\bs^1$. By using Theorem~6.2 in \cite{MoSc80}, we
conclude that $\varphi$ is a measurable coboundary. Let us mention
that another proof based on the tightness of the cocycle
$(\varphi_n)$ can also be given.

Conversely, if $\varphi$ is a measurable coboundary, then $e^{2\pi i
s \varphi} =f\circ T/f$, for $s$ real has a measurable solution, and
this implies that $\beta$ has the expansion given
by~(\ref{betaExpan}).

Therefore we obtain:
\begin{Prop} \label{exCob} If $\ell $ is a positive integer with $\ell \beta<1$,
then the cocycle $\varphi = \ell 1_{[0, \beta)} - 1_{[0, \ell
\beta)}$ is a coboundary if and only if $\beta$ satisfies
(\ref{betaExpan}).
\end{Prop}

{\bf Question:} A question is to know if the cocycle $\varphi = \ell
1_{[0, \beta)} - 1_{[0, \ell \beta)}$ is regular or not, when
$\beta$ has an expansion $\beta = \sum_{n\geq 0} b_n q_n\alpha \mod
1, {\rm \ with \ } \lim_n {|b_n| \over a_{n+1}} = 0 {\rm \ and \ }
\sum_{n\geq 0} \frac{|b_n|}{a_{n+1}} = +\infty$. (Notice that by
Theorem \ref{NSCdtion} it cannot be a coboundary.)

\vskip 3mm {\bf d= 3, $\Phi_3= (1_{[0, \beta)} - \beta, 1_{[0,
\gamma)} - \gamma,1_{[0, \delta)} - \delta)$}

\vskip 2mm We will consider $\alpha$ of non bounded type and
construct non regular cocycles (cf. \cite{Co09}). For $r \in \R$, we
denote by $\rho_r$ the translation $x \rightarrow x+r \hbox{ mod }
1$.

\begin{Th}\label{examplenonreg} Assume that $Tx=x+\alpha$ on the circle $\T$.
If $\alpha$ is not of bounded type, then there exists $\beta$ such
that $\varphi = 1_{[0,\beta)} - 1_{[0,\beta)}\circ \rho_r$ is a non
regular cocycle for $r$ in a set of full Lebesgue measure.
\end{Th}
\begin{proof} \ By a result of Merril (\cite{Me85}, Theorem 2.5 therein, see also
Theorem \ref{NSCdtion} above from \cite{GuPa06}), we know that, if
$\beta$ satisfies~(\ref{betaExpan}), then there is an uncountable
set of real numbers $s$ (so containing irrational numbers) such that
we can solve the following quasi-coboundary multiplicative equation
in $(s,\beta)$: for $s\in\R$ there exist $|c|=1$ and a measurable
function $f :\T\to \bs^1$ such that $e^{2\pi i s1_{[0,\beta)}}= c {f
/ f \circ T}$.

For this choice of $\beta$ and $s$ ($s$ is irrational), $e^{2\pi i
s(1_{[0,\beta)} - 1_{[0,\beta)}\circ \rho_r)}$ is a multiplicative
coboundary for every $r$.

For the integer valued cocycle $\psi_r=1_{[0,\beta)} - 1_{[0,\beta)}
\circ \rho_r$ we obviously have $\E(\psi_r) \subset \Z$. On the
other hand, $s \,\psi_r(x)= n(x)+F(x)-F(x+\alpha)$, with $F:X\to\R$
and $n(\cdot):X\to\Z$ measurable. Therefore $\psi_r(x) = s^{-1}
n(x)+ s^{-1} F(x)- s^{-1} F(x+\alpha)$.

It follows that the group of finite essential values over $T$ of the
cocycle $\psi_r$ is also included in the group $\frac1s\Z$ and
therefore $\overline \E(\psi_r) \subset \{0, \infty\}$.

This implies that $\psi_r$ is either non regular or a coboundary
(cf.\ Subsection \ref{prelimin}). The latter case cannot occur for a
set of values of $r$ of positive measure, because otherwise, by
Proposition \ref{cobord} below, $1_{[0,\beta)} - \beta$ is an
additive coboundary up to some additive constant $c$ (and
necessarily $c =0$, since the cocycle defined by $1_{[0,\beta)} -
\beta$ is recurrent). But this would imply that $e^{2\pi i \beta}$
is an eigenvalue of the rotation by $\alpha$, a contradiction.

Therefore the cocycle $1_{[0,\beta)} - 1_{[0,\beta)}\circ \rho_r$ is
non regular for a.e. $r \in \R$. \end{proof} \vskip 3mm
\begin{Prop}\label{cobord} Assume that $K$ is a compact connected
Abelian (monothetic) group. Let $T$ be an ergodic rotation on $K$.
Let $\varphi:K\to\R$ be a cocycle. Assume moreover, than on a set of
$g\in K$ of positive Haar measure we can find a measurable function
$\psi_g:K\to\R$ such that \begin{eqnarray} \varphi -
\varphi(g+\cdot) = \psi_g \circ T - \psi_g. \label{eq-lin}
\end{eqnarray}
Then $\varphi$ is an additive quasi-coboundary, i.e.\ $\varphi = b +
h \circ T - h$, for a measurable function $h:K\to\R$ and a constant
$b\in\R$.
\end{Prop}
\begin{proof} \ \ For $g\in K$ satisfying~(\ref{eq-lin}) and arbitrary
$s \in \R$ we have:
$${e^{2\pi i s\varphi(x)} \over e^{2\pi i s\varphi(g+x)}} = {e^{2\pi i
s\psi_g(T x)} \over e^{2\pi i s\psi_g(x)}}.$$

According to Proposition 3 in \cite{Le93}, for every $s$ there exist
$\lambda_s$ with $|\lambda_s|=1$ and a measurable function
$\zeta_s:X\to\bs^1$ such that $e^{2\pi i s\varphi} = \lambda_s \cdot
\zeta_s \circ T/ \zeta_s$. By Theorem 6.2 in \cite{MoSc80}, the
result follows.
\end{proof}

\vskip 3mm
\begin{Remark} 1) \ If $\beta$
satisfies~(\ref{betaExpan}), then either $1_{[0,\beta)} -
1_{[0,\beta)}\circ \rho_r$ is non regular or is a coboundary. We
have shown that the latter case can occur only for $r$ in a set of
zero measure. A problem is to explicit values of $r$ for which
$1_{[0,\beta)} - 1_{[0,\beta)}\circ \rho_r$ is not a coboundary.

2) If $\psi_{\beta, \frac12} := 1_{[0,\beta)} - 1_{[0,\beta)}\circ
\rho_\frac12$ is non regular, then $\psi_{[\frac12 - \beta,
\frac12)} := 1_{[0,\frac12 - \beta)} - 1_{[0, \frac12 - \beta)}\circ
\rho_\frac12$ is regular. Indeed the sum of these two cocycles is
$1_{[0,\frac12)} - 1_{[\frac12,1)}$. It can be easily shown that
this latter cocycle has non trivial quasi periods. The non
regularity of $\psi_{\beta, \frac12}$ implies that $(\psi_{[\beta,
\frac12)})_{q_n}$, the cocycle at times $q_n$, tends to 0 in
probability, so that $\psi_{[\frac12 - \beta, \frac12)}$
 has non trivial quasi periods.
\end{Remark}

\begin{Cor} \label{examplenonreg2} There are values of the parameters
$(\beta, \gamma, \delta)$ such that
$$\Phi_3= (1_{[0, \beta)} - \beta, 1_{[0, \gamma)} - \gamma,1_{[0, \delta)} - \delta)$$
is non regular.
\end{Cor}
\begin{proof} Suppose that $0 < \beta < \gamma < \delta$ and $\delta = \beta + \gamma$.
By applying the map $(y_1, y_2, y_3) \to y_1 + y_2 -y_3$, we obtain
the 1-dimensional cocycle $1_{[0, \beta)}(\cdot) - 1_{[0, \beta)}
(\cdot+ \gamma)$, which is non regular by
Theorem~\ref{examplenonreg} for a value of the parameter $\beta$
satisfying~(\ref{betaExpan}) and almost all $\gamma$.
Lemma~\ref{nonregIm} implies the non regularity of $\Phi_3$ for
these values of the parameters.
\end{proof}

Note that for $d= 2$, i.e.\ for two parameters ($\beta, \gamma)$, an
attempt to obtain a non regular cocycle is to take $\gamma = 2
\beta$ and the linear combination: $2(1_{[0, \beta)}(\cdot) - \beta)
- (1_{[0, 2\beta)} (\cdot) - 2\beta) = 1_{[0, \beta)}(\cdot) -
1_{[0, \beta)} (\cdot+ \beta)$. We obtain the cocycle discussed
above (cf.\ Proposition~\ref{exCob}) and the question previously
mentioned above is whether there are values of $\beta$ such that it
is non regular.

\vskip 3mm
\section{Application to affine cocycles}

We consider now the affine cocycle
$$\Psi_{d+1}(x):=(\psi(x),
\psi(x+\beta_1),..., \psi(x+\beta_{d})), {\rm \ where \ } \psi(x)
=\{x\} - \frac12.$$

\subsection{Reduction to a step function}

\vskip 3mm By a straightforward calculation we have the following
formula for the cocycle $\psi$: \beq\label{p1} \psi_{\qn}(x)=\qn
x+\frac{\qn(\qn-1)}{2}\alpha-\frac{\qn}{2}+M(x), \eeq where $M$ is a
(non $1$-periodic) function with values in $\Z$. It follows that,
for $\beta \in [0,1)$, \beq \label{p12} \psi_{\qn}(\{x+\beta\}) \eeq
$$=\left\{\begin{array}{lll}\psi_{\qn}(x)+\qn\beta
+(M(x+\beta)-M(x)),& \text{ if } &x+\beta<1, \\
\psi_{\qn}(x)+(\qn\beta-\qn) +(M(\{x+\beta\})-M(x)),& \text{ if }
&1\leqslant x+\beta<2.\end{array}\right.$$

We will reduce the cocycle $\Psi_{d+1}$ to step cocycles using the
group of finite essential values.

\begin{Prop} \label{reduc1}
The group $\E(\Psi_{d+1})$ includes $\Delta_{d+1}=\{(t,...,t):\: t
\in \R\}$, the diagonal subgroup of $\R^{d+1}$.
\end{Prop} \begin{proof} \ Denote
$S_{i}(x)=\rho_{\beta_i}(x)=x+\beta_{i} \ {\rm mod} \ 1$. Suppose
that ${\{q_{n_k}\beta_{i}}\}\rightarrow c_{i}$, with $c_{i}\in[0,1)$
for $i=1,\ldots, d$, and consider the measures
$$ \nu_{k}:=((\psi \times \psi\circ S_{1} \times \ldots \times
\psi \circ S_{d})_{q_{n_{k}}})_{\ast}(\mu),\;\;k\geq1.$$

Since $$\forall x,y \in [0,1), \quad |\psi_{\qnk}(x)-\psi_{\qnk}
(y)|<2 \, V(\psi) =2$$ and $\int\psi\, d\mu=0$, we have that ${\rm
Im} (\psi \times\psi\circ S_{1}\times\ldots\times \psi\circ
S_{d})_{\qnk} \subset [-2, 2]^{d+1}$, so that $\nu_k$ is
concentrated on $[-2, 2]^{d+1}$.

It follows that we can select a subsequence of $(\nu_{k})$ (still
denoted $(\nu_{k})$) which converges to a probability measure $\nu$
concentrated on $[-2,2]^{d+1}$. We will show in what kind of a
subset of $\R^{d+1}$ the support of $\nu$ is included. Consider the
image of the measure $\nu_{k}$ via
$$F:\R^{d+1}\rightarrow \R^{d}, \ \ \ F(x_{0}, \ldots, x_{d})= (x_{1}-x_{0},
\ldots,x_{d}-x_{0}).$$ In view of (\ref{p12}), we obtain
$$F\circ (\psi\times\psi\circ S_{1}\times\ldots\times \psi \circ
S_{d})_{q_{n_{k}}}(x)=(\{\qnk\beta_{1}\}+M_{1}(x), \ldots,
\{\qnk\beta_{d}\}+M_{d}(x) )$$ with $M_i(x)\in\Z$, whence
$F_{\ast}\nu_{k}$ is the measure concentrated on $(\{q_{n_k}
\beta_1\},...,\{q_{n_k}\beta_{d}\})+\Z^{d}$.

Since $\nu_{k}\rightarrow \nu$ weakly, $F_{\ast}\nu_{k}\rightarrow
F_{\ast}\nu$ (because all these measures are concentrated on a
bounded subset of $\R^{d+1}$). As ${\{q_{n_k}\beta_{i}}\}\rightarrow
c_{i}$, it follows that
$${\rm supp}\,\nu \subset \{(x_{0},\ldots,x_{d})\in\R^{d+1}:\;
x_{i}-x_{0}=c_{i}+k_{i},\; k_i\in\Z,\; i=1,\ldots,d\}.$$

The set on the right hand side of this inclusion is equal to the
union of sets of the form $\{(x, x-(c_1 + k_1),...,x-(c_d+k_d):
x\in\R\}$, hence of countably many lines parallel to the diagonal
$\Delta_{d+1}$. Moreover, the support of $\nu$ is uncountable
(because one dimensional projections of $\nu $ are absolutely
continuous measures - see \cite{LePaVo96}), whence it must be
uncountable on one of these lines. In view of
Proposition~\ref{supp}, ${\rm supp}\,\nu\subset \E(\Psi_{d+1})$ and
since $\E(\Psi_{d+1})$ is a group, we have ${\rm supp}\,\nu-{\rm
supp}\,\nu \subset \E(\Psi_{d+1})$. However, the set
$\Delta_{d+1}\cap ({\rm supp}\,\nu-{\rm supp}\,\nu)$ is uncountable,
so because $\E(\Psi_{d+1})$ is closed, we must have
$\Delta_{d+1}\subset \E(\Psi_{d+1})$ and the proof is complete.
\end{proof}

\vskip 3mm
\begin{Cor}\label{dense} $(\psi, \psi\circ S_{1}, \ldots,
\psi \circ S_{d})$ is ergodic whenever the set of accumulation
points of $(\{q_{n}\beta_{1}\}, \ldots, \{\qn\beta_{d}\})$ is dense
in $\tor^{d}$.
\end{Cor}
\begin{proof}
 From the proof of Proposition~\ref{reduc1} it follows that with
every accumulation point $ (c_{1}, \ldots, c_{d})$ of
$(\{q_{n}\beta_{1}\}, \ldots, \{\qn\beta_{d}\})$ we obtain a line
$\{(x, x-(c_1+k_1),...,x-(c_d+k_d): x\in\R\}$ (and the smallest
subgroup in which the line is included) which is included in the
group of essential values. Since the set of accumulation points is
dense and $\E(\Psi_{d+1})$ is closed, it follows that the only
possibility is that $\E(\Psi_{d+1})=\R^{d+1}$ which is equivalent to
the fact that $\Psi_{d+1}$ is ergodic.\end{proof}

\vskip 3mm By Lemma~\ref{quotient} the study of $\Psi_{d+1}$ reduces
to that of the quotient cocycle $\Psi_{d+1} +\Delta_{d+1}:
\T\to\R^{d+1}/\Delta_{d+1}$. Using the epimorphism $\R^{d+1}\ni(y_0,
..., y_d) \to (y_1 -y_0, ..., y_d - y_0)\in \R^d$ (whose kernel is
equal to $\Delta_{d+1}$), the quotient is given by the cocycle
$$\Phi_d(x) =(1_{[0, 1- \beta_j)} \, - 1 + \beta_j)_{j=1,...,d}.$$

\vskip 3mm
\subsection{Small values of $d = 1, 2, 3$ (and $\Psi_{d+1}$)}

1) $d=1,\Psi_{2} = (\psi(x), \psi(x+\beta))$

Applying Theorem \ref{reduc1} and Lemma~\ref{quotient} we can reduce
the cocycle $\Psi_{2}$ to the quotient cocycle
$\left(\Psi_{2}+\Delta_{2}\right)(x)=1_{[0, 1- \beta)} - 1 + \beta$.
We conclude using Theorem~\ref{case1} that $\Psi_{2}$ is regular
over any irrational rotation $T$.

\vskip 3mm 2) $d=2,\Psi_{3}=(\psi(x), \psi(x+\beta),
\psi({x}+\gamma))$

As above we reduce the cocycle $\Psi_{3}$ to the quotient cocycle
$\left(\Psi_{3}+\Delta_{3}\right)(x)=(1_{[0, 1- \beta)} - 1 +
\beta,1_{[0, 1- \gamma)} - 1 + \gamma)$. Recall that we have seen in
subsection \ref{123} that for $\alpha$ with bounded partial
quotients $\Psi_{3}+\Delta_{3}$ is regular and therefore the affine
cocycle is also regular when $\alpha$ has bounded partial quotients.

\vskip 3mm 3) $d=3,\Psi_{4}=(\psi(x), \psi(x+\beta),
\psi({x}+\gamma), \psi(x+\delta))$

\begin{Th} There are values of the parameters $(\beta, \gamma, \delta)$ for which the
cocycle is non regular.
\end{Th}
\begin{proof} After reduction by $\Delta_4$, the result follows from
Corollary \ref{examplenonreg2}.
\end{proof}

\subsection{Ergodicity is generic}\label{secgeneric}
We consider, as before, the cocycle $\psi(x)=\{x\}-\frac12$ and let
$S_{\beta}(x)=x+\beta$ be the rotation by $\beta \in [0,1)$ on $\T$.

\begin{Prop} The set $\{(\beta_{1}, \ldots, \beta_{d})\in\T^{d}:\: (\psi,
\psi\circ S_{\beta_1}, \ldots, \psi\circ S_{\beta_d}))\;\mbox{is
ergodic}\}$ is residual.
\end{Prop}
\begin{proof}
Using Corollary~\ref{dense}, we only need to show that the set of
$(\beta_{1}, \ldots,\beta_{d})$ for which the set of accumulation
points of $(\{\qn\beta_{1}\}, \ldots, \{\qn\beta_{d}\})_{n\geq1}$ is
dense in $\T^d$, is residual ({\em i.e.} it includes a dense
$G_{\delta}$ subset).

We take $\varepsilon>0$ , $c_{1}, \ldots, c_{d}\in[0,1)$ and
consider the sets $\widetilde{A}_{N} =
\widetilde{A}_{N}(c_{1},\ldots, c_{d}, \varepsilon) :=
\bigcup_{n=N}^{\infty}A_{n}(c_{1}, \ldots, c_{d}, \varepsilon)$,
where
\begin{eqnarray*}
A_{n} &=& A_{n}(c_{1}, \ldots, c_{d}, \varepsilon) :=\{(\beta_{1},
\ldots, \beta_{d})\in\tor^{d}:\: \|\qn\beta_{1}-c_{1}\|<\varepsilon,
\ldots, \|\qn\beta_{k}-c_{d}\|<\varepsilon\}.
\end{eqnarray*}
Clearly $\widetilde{A}_{N}$ is open and also dense. Fix $0 <
\varepsilon_\ell \to 0$. Then the set
$$\bigcap_{\ell\geq1}\bigcap_{N=1}^{\infty}
\widetilde{A}_{N}(c_{1}, \ldots, c_{d}, \varepsilon_\ell)$$ is a
dense $G_\delta$. Moreover this set equals
$$\{(\beta_{1}, \ldots, \beta_{d})\in\tor^{d}: \left(\exists
q_{n_{k}}\right)\;\;(\{q_{n_{k}}\beta_{1}\}, \ldots,
\{q_{n_{k}}\beta_{d}\})\to (c_{1}, \ldots, c_{d})\},$$ so the latter
set is also a dense $G_{\delta}$. Therefore the set
$$\bigcap_{(c_{1}, \ldots, c_{d}) \,\in \, \Q^{d} \, \cap[0,1)^{d}}
\bigcap_{\ell=1}^\infty\bigcap_{N=1}^{\infty}
\widetilde{A}_{N}(c_{1}, \ldots, c_{k}, \varepsilon_\ell)$$ is a
dense $G_{\delta}$ and the proof is complete.
\end{proof}

Now, we show that the multiple ergodicity problem has a positive
answer for a.a.\ choices of $(\beta_{1}, \ldots,\beta_{d})$. We will
need the following classical lemma of Rajchman.

\begin{Lemma} \label{LNT}
Let $(X,\cal B,\mu)$ be a probability space, $f_{n}:X\rightarrow\R$
such that $f_{n}\in L^{2}(X, \cal B, \mu)$, $\|f_{n}\|< C$, and
$f_{n}\bot f_{m}$ whenever $n\neq m$. Then $\frac{1}{n}
\sum_{k=1}^{n}f_{k}\rightarrow 0\ \ a.e.$
\end{Lemma} \begin{proof}
It follows from the assumptions that $\sum_{N=1}^\infty \|{1 \over
N^2} \sum_{k=1}^{N^2} f_k\|_2^2 \leq \sum_{N=1}^\infty {C^2 \over
N^2} < +\infty$; hence, $\lim_N \frac{1}{N^2}\sum_{k=1}^{N^2}f_{k} =
0$ a.e.

For $n \geq 1$, let $L_n := [\sqrt{n}]$. We have $L_n^2 \leq n <
(L_{n}+1)^2$ and
$$|\frac{1}{n}\sum_{k=1}^{n}f_{k}| \leq {1 \over L_n^2}
|\sum_{k=1}^{L_n^2} f_k| + 2C {L_n \over n} \underset {n \to \infty}
\longrightarrow 0, \text{a.e.}$$
\end{proof}

\begin{Prop} For every irrational rotation $Tx=x+\alpha$ on $\T$, we
have
$$\mu^{\otimes d}\{(\beta_{1}, \ldots, \beta_{d})\in\tor^{d}
:\: (\psi, \psi\circ S_{\beta_1}, \ldots, \psi\circ S_{\beta_d}))\;
\text{is}\;\; T\text{-ergodic}\}=1.$$
\end{Prop} \begin{proof}
By Corollary ~\ref{dense}, all we need to show is that the set of
$(\beta_{1}, \ldots,\beta_{d})$ for which the set of accumulation
points of $(\{\qn\beta_{1}\}, \ldots, \{q_n\beta_{d}\})_{n\geq1}$ is
dense in $\tor^{d}$, is a set of full measure. We will show more:
the set of such $d$-tuples for which $(\{\qn\beta_{1}\}, \ldots,
\{q_n\beta_{d}\})_{n\geq1}$ is uniformly distributed (mod~$1$) in
$\tor^{d}$ is of full measure.

For almost all $(\beta_{1}, \ldots, \beta_{d})$, the sequence
$(q_n\beta_{1}, \ldots, q_n\beta_{d})_{n \geq 1}$ is uniformly
distributed (mod~$1$). Indeed, by Weyl's criterium of
equidistribution (see e.g.\ \cite{KuNi}) it suffices to show that
for almost all $(\beta_{1}, \ldots, \beta_{d}) \text{ in }
\tor^{d}$, for any nontrivial character $\chi$ of $\tor^{d}$, the
Cesaro averages of the sequence $(\chi (q_n\beta_{1}, \ldots,
q_n\beta_{d}))_{n \geq 1}$ tend to zero.

We have $\chi (q_n\beta_{1}, \ldots, q_n\beta_{d}) =\exp(2\pi i(s_1
q_n\beta_{1}+ \ldots+ s_{d} q_n\beta_{d}))$ for integers
$s_{1},\ldots,s_{d}$. To conclude, we apply Lemma \ref{LNT} to
$f_{n}(x_{1}, \ldots, x_{d}):=\exp(2\pi i(\qn s_{1} x_{1} +
\ldots+\qn s_{d}x_{d}))$.
\end{proof}

\vskip 3mm The authors are grateful to M. Lema\'nczyk and E. Lesigne
for their valuable suggestions. They thank the referee for his
careful reading and his helpful remarks.

Jean-Pierre Conze \\
IRMAR, CNRS UMR 6625, University of Rennes I,\\
Campus de Beaulieu, 35042 Rennes Cedex, France\\
conze@univ-rennes1.fr

Agata Pi\c{e}kniewska\\
Faculty of Mathematics and Computer Science,\\
Nicolaus Copernicus University,\\ ul. Chopina 12/18, 87-100 Toru\'n, Poland\\
a.piekniewska@gmail.com

\end{document}